\documentclass{aims}
\usepackage{amsmath}
\usepackage{amssymb,verbatim}
  \usepackage{paralist}
  \usepackage{graphics} 
  \usepackage{epsfig} 
\usepackage{graphicx}  \usepackage{epstopdf}
 \usepackage[colorlinks=true]{hyperref}
\hypersetup{urlcolor=blue, citecolor=red}

\def\currentvolume{X}
 \def\currentissue{X}
  \def\currentyear{200X}
   \def\currentmonth{XX}
    \def\ppages{X--XX}
     \def\DOI{10.3934/xx.xx.xx.xx}

\newtheorem{thm}{Theorem}[section]
\newtheorem*{thmA}{Theorem A}
\newtheorem*{thmM}{Main Theorem}
\newtheorem{lem}[thm]{Lemma}
\newtheorem{cor}[thm]{Corollary}

\newtheorem{prop}[thm]{Proposition}

\theoremstyle{definition}
\newtheorem{defn}[thm]{Definition}
\newtheorem{example}[thm]{Example}

\newcommand{\ca}{caterpillar}

\newcommand{\R}{\mathbb{R}}
\newcommand{\iU}{U^\infty}

\newcommand{\phd}{\mathrm{PHD}}
\newcommand{\fg}{\mathrm{FG}}

\newcommand{\car}{\mathrm{CA}}

\newcommand{\orb}{\mathrm{orb}}

\newcommand{\cu}{\mathrm{CU}}

\newcommand{\F}{\mathcal{F}}

\newcommand{\oa}{{\ol{a}}}
\newcommand{\bj}{\bar j}
\newcommand{\bq}{\bar q}

\newcommand{\pr}{\mathrm{pr}}

\newcommand{\di}{\ol{\mathrm{Di}}}
\newcommand{\ol}{\overline}

\newcommand{\0}{\varnothing}
\newcommand{\sm}{\setminus}

\newcommand{\bd}{\mathrm{Bd}}

\newcommand{\ch}{\mathrm{Ch}}

\newcommand{\si}{\sigma}
\newcommand{\ta}{\theta}

\newcommand{\nin}{\not\in}

\newcommand{\hell}{\hat{\ell}}

\newcommand{\C}{\mathbb{C}}
\newcommand{\hc}{\mbox{$\mathbb{\widehat{C}}$}}

\newcommand{\bbd}{\mbox{$\mathbb{D}$}}
\newcommand{\disk}{\mathbb{D}}

\newcommand{\uc}{\mathbb{S}^1}

\newcommand{\lam}{\mathcal{L}}
\newcommand{\M}{\mathcal{M}}
\newcommand{\Z}{{\mathbb{Z}}}

\renewcommand\le{\leqslant}
\renewcommand\ge{\geqslant}

\title[Laminations from the Main Cubioid]
{Laminations from the Main Cubioid}

\author[Alexander Blokh, Lex Oversteegen, Ross Ptacek and Vladlen Timorin]{}

\subjclass{Primary 37F20; Secondary 37C25, 37F10, 37F50}

\keywords{Complex dynamics; Julia set; Mandelbrot set}

\email[Alexander~Blokh]{ablokh@math.uab.edu}
\email[Lex~Oversteegen]{overstee@math.uab.edu}
\email[Ross~Ptacek]{rptacek@uab.edu}
\email[Vladlen~Timorin]{vtimorin@hse.ru}

\thanks{The first and the third named authors were partially
supported by NSF grant DMS--1201450}

\thanks{The second named author was partially  supported
by NSF grant DMS-0906316}
\thanks{The fourth named author was partially supported by a subsidy granted to
the HSE by the  Government of the Russian Federation for the implementation
of the Global Competitiveness Program and by RFBR grant 16-01-00748à.
}

\thanks{$^*$ Corresponding author: Alexander Blokh}

\begin{document}

\maketitle

\centerline{\scshape Alexander Blokh and Lex Oversteegen$^*$}
\medskip
{\footnotesize
 \centerline{ Department of Mathematics}
   \centerline{University of Alabama at Birmingham}
   \centerline{Birmingham, AL 35294, USA}
} 

\medskip

\centerline{\scshape Ross Ptacek and Vladlen Timorin}
\medskip
{\footnotesize
 \centerline{ Faculty of Mathematics,
Laboratory of Algebraic Geometry and its Applications}
   \centerline{National Research University Higher School of Economics}
   \centerline{Vavilova St. 7, 112312 Moscow, Russia}
}

\bigskip

 \centerline{(Communicated by the associate editor name)}

\begin{abstract}
Polynomials from the closure of the principal hyperbolic domain of
the cubic connectedness locus have some specific properties, which
were studied in a recent paper by the authors. The family of (affine
conjugacy classes of) all polynomials with these properties is
called the Main Cubioid. In this paper, we describe a combinatorial
counterpart of the Main Cubioid --- the set of invariant laminations
that can be associated to polynomials from the Main Cubioid.
\end{abstract}






\section{Introduction}\label{s:intro}

\subsection{Motivation}

The \emph{complex quadratic family} is the family of all polynomials
$P_c(z)=z^2+c$ (any quadratic polynomial is affinely conjugate to
some $P_c$). An important role in studying this family is played
by the \emph{connectedness locus $\M_2$} (also called the
\emph{Mandelbrot set}) consisting of all $c$ such that the
\emph{Julia set} $J_{P_c}$ is connected. The central
part of $\M_2$ is the Principal Hyperbolic Domain ${\rm PHD}_2$,
i.e., the set of numbers $c\in \C$ such that
$P_c$ has an attracting fixed point. Its closure $\car$ is called
the \emph{Main Cardioid \emph{(}of the Mandelbrot set\emph{)}}. A combinatorial
model $\M^c_2$ of $\M_2$, due to Thurston \cite{thu85}, implies a
combinatorial model $\car^c$ of $\car$; we call $\car^c$ the
\emph{Combinatorial Main Cardioid}.

Similarly, in degree $d$ one can consider the space of affine
conjugacy classes of degree $d$ polynomials (in the quadratic
family, we made an explicit choice of a representative polynomial).
The  \emph{degree $d$ connectedness locus} (also called the
\emph{degree $d$ Mandelbrot set}) $\M_d$ is the space of all degree
$d$ affine conjugacy classes, whose polynomials have connected Julia
sets (equivalently, all critical points have bounded orbits). In
what follows, by the \emph{class} of a polynomial we will always
mean the affine conjugacy class. Given a polynomial $P$, we denote
its class by $[P]$. The \emph{Principal Hyperbolic Domain} ${\rm
PHD}_d$ in $\M_d$ consists of all classes of hyperbolic polynomials
with Julia set homeomorphic to a circle. Equivalently, $[P]$ is in
${\rm PHD}_d$ if all critical points of $P$ are in the immediate
basin of attraction of some attracting fixed point. An important
question then is to describe the set of all classes of polynomials
that belong to the closure $\ol{\rm PHD}_d$ of ${\rm PHD}_d$. Here
we address this question for $d=3$ (i.e., in the \emph{cubic} case).

The key object in Thurston's combinatorial model is the notion of a
\emph{lamination} (full definitions are provided in
Section~\ref{s:prelim}, see Definitions~\ref{d:lam},
~\ref{d:si-inv-lam}, ~\ref{geolam}, and ~\ref{geolaminv}, while in
Subsection~\ref{ss:lami} we will only give loose descriptions).
Laminations provide combinatorial models for connected polynomial
Julia sets. Thurston's work \cite{thu85} can be seen as consisting
of two parts: he defined a set of laminations that are good
candidates to be models of quadratic Julia sets, and then he
described the set of such laminations, which led to a combinatorial
model for the Mandelbrot set.

In this paper, we will make a similar first step for $\ol{{\rm
PHD}}_3$ with the hope that our classification eventually yields a
suitable combinatorial model. Polynomials from $\ol{{\rm PHD}}_3$
satisfy certain dynamical properties.  These dynamical properties in
turn force the corresponding laminational models to have certain
properties. We will consider all laminations that satisfy these
properties and provide a simple classification of these laminations
which revolves around a reduction to laminations with model
polynomials from $\ol{{\rm PHD}}_2$.

The cubic Mandelbrot set $\M_3$ or its parts have been studied
before. P. Lavaurs in his thesis titled
\emph{Syst\`emes dynamiques holomorphes, Explosion de points
p\'eriodiques paraboliques} (Universit\'e
Paris-Sud, Orsay, 1989) proved that $\M_3$ is not locally
connected. Epstein and Yampolsky \cite{EY99} proved that the
bifurcation locus in the space of real cubic polynomials is not
locally connected either. This makes defining a combinatorial model
of $\M_3$ very delicate. Buff and Henriksen \cite{BH01} presented
copies of quadratic Julia sets, including Julia sets which are not
locally connected, in slices of $\M_3$; however, these copies are
disjoint from the closure of $\mathrm{PHD}_3$. Moreover, McMullen
\cite{mcm07} has shown that slices of $\M_3$ contain lots of copies
of $\M_2$. In addition, Gauthier \cite{gau14} has shown that $\M_3$
contains copies of $\M_2\times \M_2$. Observe also that in his
thesis titled \emph{Local connectivity in a family of cubic polynomials}
(Cornell University, 1992) D. Faught considered the slice $A$ of $\M_3$
consisting of polynomials with a fixed critical point and showed
that $A$ contains countably many homeomorphic copies of $\M_2$ and
is locally connected everywhere else. In particular, $A$ intersects
the boundary of $\mathrm{PHD}_3$ along a Jordan curve.

Roesch \cite{R06} generalized Faught's results to higher degrees.
Zakeri \cite{Z99} described some important Jordan curves in the
boundary of $\mathrm{PHD}_3$, whose points are represented by
polynomials with both critical points on the boundary of the same
Siegel disk. Milnor and Poirier \cite{MP92} gave a classification of hyperbolic
components in $\M_d$; in particular, he proved 
that the topology of many hyperbolic components can be reduced
to that of $\mathrm{PHD}_3$. In a recent paper \cite{PT09}, Petersen
and Tan Lei introduced an analytic coordinate system on
$\mathrm{PHD}_3$ that reflects dynamical properties of the
corresponding polynomials. The authors planned a sequel \cite{PRT}
to this paper, in which the boundary of $\mathrm{PHD}_3$ would be discussed.
After the main results of our paper had been obtained, we discovered
that our work may have some overlap with \cite{PRT}.

\subsection{Introduction to Laminations}\label{ss:lami}
Thurston \cite{thu85} gave a combinatorial model for the entire
Mandelbrot set. It has been conjectured that the Mandelbrot set is
homeomorphic to Thurston's model; in fact, this conjecture is
equivalent to local connectivity of $\M_2$. Although a global
homeomorphism is not known, some parts of the Mandelbrot set can be
shown to be homeomorphic to the corresponding parts of the model.
For example, $\car$ is homeomorphic to $\car^c$, and both sets are
homeomorphic to the closed disk (see, e.g., \cite{cg}).

For higher degree Mandelbrot sets $\M_d$ even conjectural
models are missing.
To begin with, it is natural to model $\ol{\phd}_d$. As a first step
to solving this problem in the cubic case, we study individual
polynomials from $\ol{\phd}_3$. Similar to Thurston \cite{thu85}, we
use \emph{laminations}.

We will write $\C$ for the plane of complex numbers, $\widehat\C$ for
the Riemann sphere, and $\disk=\{z\in\C\,|\, |z|<1\}$ for the open unit disk.
A lamination is a closed equivalence relation $\sim$ on $\uc=\{z\in\C\,|\,|z|=1\}$,
whose classes are finite sets, such that the convex hulls of
different classes have disjoint relative interiors, see Definition \ref{d:lam}.
A lamination is ($\si_d$-)invariant if classes map to classes under
$\si_d:\uc\to\uc$, $z\mapsto z^d$ in a covering fashion. See
Definition \ref{d:si-inv-lam}, which makes this more precise.
If a polynomial $P$ has a locally connected Julia set $J$, then
there is a lamination $\sim_P$ identifying pairs of angles if the
corresponding external rays land at the same point.
The quotient
$J_{\sim_P}=\uc/{\sim_P}$ is homeomorphic to $J$, and the
self-mapping $f_{\sim_P}$ of $J_{\sim_P}$ induced by $\si_d$ is
conjugate to $P|_{J_P}$; the map $f_{\sim_P}$ and the set
$J_{\sim_P}$ are called a \emph{topological polynomial} and a
\emph{topological Julia set}, respectively. Laminations can play a
role for some polynomials whose Julia sets are not locally
connected. Then $P|_{J_P}$ and $f_{\sim_P}|_{J_{\sim_P}}$ are not
conjugate, however, they are semiconjugate by a monotone map (a continuous map,
whose fibers are continua). Topological Julia sets and polynomials
make sense for any $\si_d$-invariant lamination, not just those described
above.

It is very useful to associate to each lamination $\sim$ important
geometric objects defined below.
We will identify $\uc$ with $\R/\Z$.
For a pair of angles $a$, $b\in\R/\Z$, we will write $\ol{ab}$ for
the chord (a straight line segment in $\C$) connecting the points
of $\uc$ corresponding to angles $a$ and $b$.
If $G$ is the convex hull $\ch(G')$ of some closed set $G'\subset\uc$,
then we write $\si_d(G)$ for the set $\ch(\si_d(G'))$.
The set $G'=G\cap\uc$ is called the \emph{basis} of $G$.
The boundary of $G$ will be denoted by $\bd(G)$.

\begin{defn}[Leaves]\label{d:geola}
If $A$ is a $\sim$-class, call a chord $\ol{ab}$ in $\bd(\ch(A))$ a
\emph{leaf} of $\sim$. All points of $\uc$ are also called
(\emph{degenerate\emph{)} leaves}. The family $\lam_\sim$ of all
leaves of $\sim$ is called the \emph{geo-lamination {\rm (}geometric
lamination, or geodesic lamination{\rm )} generated by $\sim$}. The
union of all leaves of $\lam_\sim$ is denoted by $\lam^+_\sim$.
\end{defn}

In general, collections of leaves with properties similar to those
of collections $\lam_\sim$ are also called \emph{invariant
geo-laminations}, see Definition \ref{geolaminv} for a more precise
formulation. In fact, it is these collections that Thurston
introduced and studied in \cite{thu85}. Thus, laminations (and their
geo-laminations) on the one hand, and geo-laminations in general, on
the other hand, are studied in complex dynamics. The motivation of
this can be as follows. The direct association between polynomials
and laminations was described above. However, not every polynomial
$P$ can be directly associated to a suitable lamination. Hence it is
natural to consider polynomials close to $P$ for which such
association is possible and then associate to $P$ the appropriately
defined limit of their laminations. To define such a limit one has
to consider geo-laminations. However, limits of geo-laminations,
which were associated to laminations, are not easily associated to
laminations. This motivates the usage of ``abstract''
geo-laminations, i.e., geo-laminations not associated with any
lamination.

\begin{defn}[Gaps]\label{d:gaps-i}
Let $\lam$ be an invariant geo-lamination, e.g., we may have
$\lam=\lam_\sim$ for some invariant lamination $\sim$. The closure
in $\C$ of a non-empty component of $\disk\sm \lam^+$ is called a
\emph{gap} of $\lam$. \emph{Edges of a gap $G$} are defined as
leaves of $\lam$ on the boundary of $G$.
A gap is said to be \emph{finite $($infinite$)$} if its basis is finite
(infinite). Infinite gaps of $\lam$ are also called \emph{Fatou
gaps}.

The map $\si_d:G'\to\si_d(G')$ extends to $\bd(G)$ as a composition
of a monotone map and a covering map of some degree $m$. Then $m$ is
called the \emph{degree} of $G$. A Fatou gap $G$ is \emph{periodic
$($of period $n)$} if the interiors of the sets $G$, $\si_d(G)$,
$\dots$, $\si_d^{n-1}(G)$ are disjoint while $\si_d^n(G)=G$. Such a
gap $G$ is said to be a periodic \emph{Siegel} gap if the degree of
$\si_d^n|_G$ is $1$.
If the degree of $G$ is $2$, then the gap $G$ is said to be \emph{quadratic}.
If the period of $G$ is $1$, then $G$ is said to be \emph{invariant}.
\end{defn}

\subsection{The cubioid}
Our studies are based on Theorem A.

\begin{thmA}[\cite{bopt13}]
If $[P]\in\ol{\phd}_d$, then $P$ has a fixed non-repelling point, no
repelling periodic cutpoints, and at most one non-repelling periodic
point with multiplier different from 1.
\end{thmA}

A polynomial $P$ with $[P]\in \ol{\phd}_3$ has at most two
non-repelling cycles. One of them must be a fixed point (as we
approximate $P$ with polynomials $g$, whose classes belong to
$\phd_3$, the attracting fixed points of $g$ converge to a
non-repelling fixed point of $P$). If there is a non-repelling cycle
of period greater than 1, then by Theorem A this cycle must have
multiplier 1. Moreover, all fixed neutral points of $f$ but one must
have multiplier 1. In this paper, we will consider all polynomials
which satisfy the conclusions of Theorem A.

\begin{defn}\label{d:cu}
The \emph{main cubioid} $\cu$ is defined as the set of all classes of cubic polynomials
that have a fixed non-repelling point, no repelling
periodic cutpoints and at most one non-repelling periodic point
with multiplier different from 1.
\end{defn}

If $[P]\in\cu$ and $J_P$ is locally connected, then
Definition~\ref{d:cu} forces the corresponding lamination $\sim_P$
to have certain properties (see Lemma~\ref{l:cu}). By an edge of a
$\sim_P$-class, we mean an edge of the convex hull of that class. If
an edge is periodic, its vertices may have a larger period than the
period of the edge (for instance a class with vertices
$\frac14,\frac34$ is fixed under $\si_3$ but its vertices are of
period two). Periodic gaps/leaves of a geo-lamination $\lam$ are
naturally associated with a rotation number (unless these are
periodic Fatou gaps of degree greater than 1). They are said to be
\emph{rotational} if the rotation number is not zero. A thorough
treatment is given in Section~\ref{s:prelim}.

\begin{lem}\label{l:cu}
If a polynomial $P$ with $[P]\in \cu$ has locally connected Julia set, then the lamination
$\sim_P$ has at most one rotational periodic set \emph{(}hence this set must be fixed\emph{)}.
Moreover, each periodic non-degenerate $\sim_P$-class $G$ has a
cycle of edges of vertex period $n$, at which a Fatou gap of period $n$ is
attached to $G$.
\end{lem}

\begin{proof}
Let $\bf g$ be a rotational periodic set of $\sim_P$, and let
$\psi_P$ denote the semi-conjugacy between $\si_3:\uc\to\uc$ and
$P:J_P\to J_P$ obtained as the composition of the quotient map from
$\uc$ to $\uc/\sim_P$ and the natural conjugacy between $f_{\sim_P}$
and $P|_{J_P}$. If $\bf g$ is finite, it maps under the
semiconjugacy $\psi_P$ to a periodic point $x\in J_P$, which must be
a cutpoint of $J_P$. By Theorem A, the point $x$ cannot be
repelling. Hence, and since $\bf g$ is rotational, the point $x$ is
parabolic with multiplier different from 1. If $\bf g$ is infinite,
then by definition $\bf g$ is a periodic Siegel gap which, since
$J_P$ is locally connected, corresponds to a periodic Siegel domain
of $P$ which contains a periodic Siegel point. Now from Theorem A it
follows that $\bf g$ is unique. Indeed, by the above each rotational
periodic set $\bf g$ of $\sim_P$ generates a periodic non-repelling
point of multiplier distinct from $1$. Yet by Theorem A the
polynomial $P$ has at most one periodic non-repelling point of
multiplier distinct from $1$. Hence there is at most one rotational
set $\bf g$. This implies the first claim of the lemma. Since each
periodic non-degenerate $\sim_P$-class $G$ corresponds to a
parabolic periodic point, the second claim follows.
\end{proof}

\begin{defn}[Combinatorial Main Cubioid]
\label{d:cubioid}
Define the \emph{Combinatorial Main Cubioid} $\cu^c$ as the family of all
cubic laminations $\sim$ with at most one rotational periodic set
(the set must be fixed because if its period were greater than $1$,
then there would be at least two such sets) and such that each
periodic non-degenerate $\sim$-class $G$ has a cycle of edges of
vertex period $n$ at which Fatou gaps of period $n$ are attached to
$G$.
\end{defn}

\subsection{Main results}\label{ss:mainre}

In this paper we classify all laminations in $\cu^c$. In doing so it
is important to keep in mind that any cubic lamination either has
two critical sets of degree 2 (so that either set maps two-to-one on
its image) or one critical set of degree 3 (which maps onto its
image three-to-one).

To state our results, we need to discuss invariant quadratic gaps of
cubic laminations. Let $U$ be such a gap. We show that the
gap $U$ has a unique \emph{major edge} (or simply \emph{major})
$M=\ol{ab}$ such that the arc $(a, b)$ is of length at least
$\frac13$ and contains no points of $\ol U$. It turns out that $M$
can be either periodic (then $U$ is said to be of \emph{periodic
type}) or critical (then $U$ is said to be of \emph{regular critical
type}).

In either case there is a specific \emph{canonical lamination
$\sim_U$} associated to $U$. In the regular critical case $\sim_U$
is defined as follows: $a\sim_U b$ if there is $N\ge 0$ such that
$\si_3^N(a)$ and $\si_3^N(b)$ are endpoints of the same edge of $U$,
and the set $\{\si_3^i(a),\si_3^i(b)\}$ is not separated by $U$ for
$i=0$, $\dots$, $N-1$. Basically, this means that there is a gap $V$
attached to $U$ along its major and folding on top of $U$ under
$\si_3$; the rest of $\sim_U$ consists of well-defined pullbacks of
$V$, which accumulate to points of $\uc$.

Now, let $U$ be of periodic type. Then one can define a periodic
quadratic gap $V$ attached to $U$ along its major $M_U=M$. The gap
$V$ has the same period $n$ as $M$ and can be defined as follows.
Take all points $a\in \uc$ such that, for any $i\ge 0$, the point
$\si_3^i(a)$ is separated from $U$ by $\si^i_3(M)$ or is an endpoint
of $\si_3^i(M)$. Then $V$ is the convex hull of the set of all such
points $a$. We call $V=V(U)$ the \emph{vassal} gap of $U$. Now,
define the lamination $\sim_U$ as follows: $a\sim_U b$ if there
exists $N\ge 0$ such that $\si_3^N(a)$ and $\si_3^N(b)$ are
endpoints of the same edge of $U$ or the same edge of $V$, and the
chord $\ol{\si_3^i(a)\si_3^i(b)}$ is disjoint from $U\cup V$ for
$i=0$, $\dots$, $N-1$. Basically, $\sim_U$ includes $U$, the orbit
of $V$, and the rest of $\sim_U$ consists of well-defined pullbacks
of $V$ which accumulate to points of $\uc$.

Let us now give a heuristic description of laminations from $\cu^c$.
Namely, a lamination $\sim$ from $\cu^c$ can be thought of as a
result of an at most two-step process. First, an invariant quadratic
Fatou gap $U$ is created, together with its canonical lamination
$\sim_U$. This could end the whole process. However it could also
happen that afterwards the gap $U$ is (\emph{weakly\emph{)} tuned}
(see Definition~\ref{d:tunegap-i} and Definition~\ref{d:wtunegap-i})
by a quadratic lamination from the Combinatorial Main Cardioid
(since $\si_3|_{U'}$ is semiconjugate to $\si_2$ by a map collapsing
all edges of $U$, it is easy to define tuning in this setting). We
show that basically this mechanism describes all laminations from
$\cu^c$.  Although the classification given in the Main Theorem
is the main result, we do obtain as a corollary that
Definition~\ref{d:cubioid} implies an even stronger condition on
attached Fatou gaps.

\begin{cor}
\label{c:othercu-i}
A lamination $\sim$ belongs to $\cu^c$ if and only if it has at most
one rotational periodic $($hence fixed$)$ set and, for each leaf
$\ell$ of $\sim$ of vertex period $n$, there is a Fatou gap of
period $n$ attached to $\ell$.
\end{cor}

To give the precise statement of our main result, we have to
introduce a few definitions. The first one relates different
laminations.

\begin{defn}\label{d:co-exist-i}
Let $G$ be a gap of some lamination. A lamination $\sim$ {\em
coexists with} $G$ if every leaf of $\sim$ intersecting an edge
$\ell$ of $G$ in $\disk$ coincides with $\ell$. A lamination $\sim$
{\em coexists} with a lamination $\simeq$ if no leaf of $\sim$
intersects a leaf of $\simeq$ in $\disk$ unless the two leaves
coincide.
\end{defn}

Now we can formalize some situations in which one lamination
modifies another one.

\begin{defn}\label{d:tuning-i}
Suppose that two laminations, $\sim$ and $\simeq$, are given.

\begin{enumerate}

\item Say that $\sim$ \emph{tunes} $\simeq$ if
$\lam_\sim\supset \lam_\simeq$.

\item Let $G$ be a Fatou gap of some lamination. If
all edges of $G$ are leaves of $\sim$ then we say that $\sim$
\emph{tunes} the gap $G$.

\end{enumerate}

\end{defn}

The terminology can be better understood if we think of $\lam_\sim$
as being obtained by adding new leaves to $\lam_\simeq$; since these
leaves can only be added \emph{inside} gaps of $\lam_\sim$, we can
think of this as \emph{tuning} of gaps of $\sim$ which explains the
terminology.

Given a periodic Fatou gap $G$ of some unspecified lamination,
consider a map $\psi_G:\bd(G)\to \uc$ that collapses all edges of
$G$ to points; $\psi_G$ maps $\bd(G)$ onto the unit circle $\uc$. If
$G$ is periodic of period $n$, then $\psi_G:\bd(G)\to \uc$
semiconjugates $\si_d^n|_{\bd(G)}$ to either irrational rotation (if
$G$ is a Siegel gap) or to $\si_k$ (if $G$ is of degree $k>1$). Now,
suppose that $\sim$ tunes a periodic quadratic gap $G$. Then leaves
of $\lam_\sim$ contained in $G$ map under $\psi_G$ to chords of
$\ol{\disk}$. In Section~\ref{s:qgaps}, we show that these chords
can be viewed as leaves of a geo-lamination generated by an
invariant lamination; this lamination is denoted by $\psi_G(\sim)$.

\begin{defn}\label{d:tunegap-i}
Suppose that $\sim$ tunes a periodic quadratic gap $G$. Then we say
that $\sim$ \emph{tunes $G$ according to the lamination
$\psi_G(\sim)$}. If $G$ is a gap of a lamination $\simeq$, and
$\sim$ tunes $\simeq$, then for brevity we also say that $\sim$
\emph{tunes $\simeq$ according to the lamination $\psi_G(\sim)$} (in
general, in this last case the behavior of $\sim$ outside $G$, even
though compatible with $\simeq$, is not completely defined by the
way $\sim$ tunes $G$). Clearly, distinct laminations can tune the
same quadratic invariant gap.
\end{defn}

A weaker (and finer) case of tuning is described in the next
definition. Let us emphasize that in it, $\sim$ only coexists with a
periodic quadratic gap $U$ (i.e., $U$ is not necessarily a gap of
$\sim$).

\begin{defn}\label{d:wtunegap-i}
If $\sim$ coexists with a periodic quadratic gap $U$ of some unspecified
lamination, and
$\psi_U(\sim)$ coincides with a quadratic lamination $\asymp$, then
we say that $\sim$ \emph{weakly tunes $U$ according to the
lamination $\asymp$}. If $U$ is a periodic gap of a lamination
$\simeq$ and $\sim$ coexists with $\simeq$, then for brevity we also
say that $\sim$ \emph{weakly tunes $\simeq$ $($on $U$$)$ according
to the lamination $\asymp$}.
\end{defn}

We are ready to state our main theorem. Recall the vassal gap $V(U)$
was defined in the second paragraph of Subsection~\ref{ss:mainre}.

\begin{thmM}
\label{t:cormin-spec-i} Let $\sim$ be a non-empty lamination from
$\cu^c$. Then there exists an invariant quadratic gap $U$ for which
$(1)$ or $(2)$ takes place.
\begin{enumerate}
\item The lamination $\sim$ coexists with the canonical lamination
$\sim_U$ and weakly tunes $\sim_U$ on $U$ according to a quadratic
lamination $\asymp$ from $\car^c$ so that edges of $U$ are not
leaves of $\sim$. Moreover, $U$ can be chosen to be of regular
critical type $($and, if $\sim$ is not canonical, then $U$ must be of
regular critical type$)$.
\item The lamination $\sim$ tunes the canonical lamination
$\sim_U$ according to a quadratic lamination $\asymp$ from $\car^c$
$($possibly empty$)$, and if $U$ is of periodic type, then the
vassal gap $V(U)$ is a gap of $\sim$.
\end{enumerate}
\end{thmM}

The rest of the paper is structured as follows. In
Section~\ref{s:prelim}, we give precise definitions of laminations
and related terminology and state technical results, including
Theorem~\ref{t:fxpt} from \cite{bfmot10}, which provides a method for
finding fixed objects in laminations. Section~\ref{s:qgaps}
investigates the structure of quadratic invariant gaps and the
related notion of canonical laminations of quadratic invariant gaps,
upon which the classification of cubioidal laminations rests. Tuning
and weak tuning, processes which insert laminations into infinite
gaps of other laminations, are also developed there. Invariant
rotational sets play a prominent role in the definition of the
Combinatorial Main Cubioid.  In Section~\ref{s:rotgaps}, we define
and study various canonical laminations.  All of these notions are
combined in Section~\ref{s:descricu}, where the full classification
of $\cu^c$ in terms of (weak) tuning quadratic invariant gaps is
given.

\section*{Acknowledgements}
During the work on this project, the fourth named author (V.T.) has visited
Max Planck Institute for Mathematics (MPIM), Bonn. During the first
revision of the paper the first named author (A.B.) was also visiting  Max
Planck Institute for Mathematics (MPIM), Bonn. A.B. and V.T. are very
grateful to MPIM for inspiring working conditions. We all would also
like to thank the referee for useful remarks.

\section{Preliminaries}\label{s:prelim}

Let $a$, $b\in \uc$. By $[a, b]$, $(a, b)$, etc, we mean the closed,
open, etc \emph{positively oriented} circle arcs from $a$ to $b$,
and by $|I|$ the length of an arc $I$ in $\uc$ normalized so that
the length of $\uc$ is $1$. In this section, we will introduce
classic results due to Douady and Hubbard \cite{hubbdoua85a, hubbdoua85b} and
Thurston \cite{thu85}. These results allow one to study connected
filled-in Julia sets by means of studying their complements in the
complex plane.

\subsection{Laminations}\label{ss:lam}
For a compactum $X\subset\C$, let $\iU(X)$ be the unbounded component of
$\hc\sm X$ containing infinity. If $X$ is connected, there exists a
Riemann mapping $\Psi_X:\hc\sm\ol\bbd\to \iU(X)$; we always normalize it
so that $\Psi_X(\infty)=\infty$ and $\Psi'_X(z)$ tends to a positive
real limit as $z\to\infty$.

Consider a polynomial $P$ of degree $d\ge 2$ with Julia set $J_P$
and filled-in Julia set $K_P$. Extend $z^d: \C\to \C$ to a map
$\ta_d$ on $\hc$. If $J_P$ is connected, then
$\Psi_{J_P}=\Psi:\C\sm\ol\bbd\to \iU(K_P)$ is such that $\Psi\circ
\ta_d=P\circ \Psi$ on the complement of the closed unit disk
\cite{hubbdoua85a, hubbdoua85b, miln00}. If $J_P$ is locally connected, then $\Psi$
extends to a continuous function
$$
\ol{\Psi}: {\hc\sm\bbd}\to
\ol{\hc\setminus K_P},
$$
and $\ol{\Psi} \circ\,\ta_d=P\circ\ol{\Psi}$ on the complement of
the open unit disk; thus, we obtain a continuous surjection
$\ol\Psi\colon\bd(\bbd)\to J_P$ (the \emph{Carath\'eodory loop}, see
\cite{car913}).

Let $J_P$ be locally connected, and set $\psi=\ol{\Psi}|_{\uc}$.
Following Thurston \cite{thu85} (see also \cite{hubbdoua85a, hubbdoua85b}), define
an equivalence relation $\sim_P$ on $\uc$ by $x \sim_P y$ if and
only if $\psi(x)=\psi(y)$, and call it the ($\si_d$-invariant) {\em
lamination of $P$} (since $\Psi$ defined above semiconjugates
$\ta_d$ and $P$, the map $\psi$ semiconjugates $\si_d$ and
$P|_{J(P)}$ which implies that $\sim_P$ is invariant). Equivalence
classes of $\sim_P$ are pairwise \emph{unlinked}: their Euclidian
convex hulls are disjoint. The topological Julia set
$\uc/\sim_P=J_{\sim_P}$ is homeomorphic to $J_P$, and the
topological polynomial $f_{\sim_P}:J_{\sim_P}\to J_{\sim_P}$ is
topologically conjugate to $P|_{J_P}$. One can extend the conjugacy
between $P|_{J_{P}}$ and $f_{\sim_P}:J_{\sim_P}\to J_{\sim_P}$ to a
conjugacy on the entire plane.

An equivalence relation $\sim$ on the unit circle, with similar
properties to those of $\sim_P$ above,  can be introduced abstractly
without any reference to the Julia set of a complex polynomial (see
\cite{blolev02a}).

\begin{defn}[Laminations]\label{d:lam}

An equivalence relation $\sim$ on the unit circle $\uc$ is called a
\emph{lamination} if it has the following properties:

\noindent (E1) the graph of $\sim$ is a closed subset in $\uc \times
\uc$;

\noindent (E2) if $t_1\sim t_2\in \uc$ and $t_3\sim t_4\in \uc$, but
$t_2\not \sim t_3$, then the open straight line segments in $\C$
with endpoints $t_1, t_2$ and $t_3, t_4$ are disjoint;

\noindent (E3) each equivalence class of $\sim$ is finite. 
\end{defn}

Recall that, for a closed set $A\subset \uc$, we denote its convex hull by
$\ch(A)$. Then by an \emph{edge} of $\ch(A)$ we mean a closed
subsegment of the straight line connecting two points of the unit
circle which is contained in the boundary of $\ch(A)$. By an
\emph{edge} of a $\sim$-class we mean an edge of the convex hull of
that class.

\begin{defn}[Laminations and dynamics]\label{d:si-inv-lam}
A lamination $\sim$ is called ($\si_d$-){\em invariant} if:

\noindent (D1) $\sim$ is {\em forward invariant}: for a class $\mathbf{g}$,
the set $\si_d(\mathbf{g})$ is a class too;

\noindent (D2) for any $\sim$-class $\mathbf{g}$, the map
$\si_d:\mathbf{g}\to\si_d(\mathbf{g})$ extends to $\uc$ as an
orientation preserving covering map
such that $\mathbf{g}$ is the full preimage of $\si_d(\mathbf{g})$ under this
covering map.

\end{defn}

Definition~\ref{d:si-inv-lam} (D2) has an equivalent version. Given
a closed set $Q\subset \uc$, a (positively oriented) {\em hole $(a,
b)$ of $Q$ $($or of $\ch(Q))$} is a component of $\uc\sm Q$. Then
(D2) is equivalent to the fact that for a $\sim$-class $\mathbf{g}$
either $\si_d(\mathbf{g})$ is a point or, for each positively
oriented hole $(a, b)$ of $\mathbf{g}$, the positively oriented arc
$(\si_d(a), \si_d(b))$ is a hole of $\si_d(\mathbf{g})$. From now
on, we assume in this paper that, unless stated otherwise, $\sim$ is
a $\si_d$-invariant lamination.

Given $\sim$, consider the \emph{topological Julia set}
$\uc/\sim=J_\sim$ and the \emph{topological polynomial}
$f_\sim:J_\sim\to J_\sim$ induced by $\si_d$. Using Moore's Theorem,
embed $J_\sim$ into $\C$ and extend the quotient map
$\pr_\sim:\uc\to J_\sim$ to $\C$ with the only non-trivial fibers
being the convex hulls of non-degenerate $\sim$-classes. A
\emph{Fatou domain} of $J_\sim$ (or of $f_\sim$) is a bounded
component of $\C\sm J_\sim$. If $U$ is a periodic Fatou domain of
$f_\sim$ of period $n$, then $f^n_\sim|_{\bd(U)}$ is conjugate
either to an irrational rotation of $\uc$ or to $\si_k$ with some
$k>1$ \cite{blolev02a}. In the case of irrational rotation, $U$ is
called a \emph{Siegel domain}.
The complement of the unbounded component of $\C\sm J_\sim$ is called
the \emph{filled-in topological Julia set} and is denoted by
$K_\sim$.
Equivalently, $K_\sim$ is the union of $J_\sim$ and its bounded Fatou domains.
If the lamination $\sim$ is fixed, we may omit $\sim$ from
the notation. By default, we consider $f_\sim$ as a self-mapping of
$J_\sim$. In what follows, for a collection $\mathcal R$
of sets, denote the union of all sets from $\mathcal R$ by $\mathcal R^+$.

In the Introduction (Definition \ref{d:geola}), we defined leaves of
a lamination $\sim$ and the geo-lamination $\lam_\sim$ associated
with $\sim$. Extend $\si_d$ (keeping the notation) linearly over all
\emph{individual chords} in $\ol{\disk}$, in particular, over leaves
of $\lam_\sim$. Note that even though the extended $\si_d$ is not
well defined on the entire disk, it is well defined on
$\lam^+_\sim$.

Recall that for a gap/leaf $U$ we denote $U\cap \uc$ by $U'$. A
gap/leaf $U$ of $\lam_\sim$ is said to be
\emph{{\rm(}pre{\rm)}periodic} if $\si_d^{m+k}(U')=\si_d^m(U')$ for
some $m\ge 0$, $k>0$. If $m$ can be chosen to be $0$, then $U$ is
called \emph{periodic}, otherwise $U$ is called \emph{preperiodic}
(hence, preperiodic implies non-periodic). Also, by a (pre)periodic
gap/leaf we mean gap/leaf which is either periodic or preperiodic. A
\emph{Fatou gap} is the $\pr_\sim$-preimage of the closure of a
Fatou domain.  Similarly, a \emph{Siegel gap} is the
$\pr_\sim$-preimage of a Siegel domain. Equivalently, these are gaps
with infinite bases. By \cite{kiwi02}, a Fatou gap $G$ is
(pre)periodic under $\si_d$.

\begin{defn}[Critical leaves and gaps]\label{d:crit}
A leaf of a lamination $\sim$ is called \emph{critical} if its
endpoints have the same image under $\si_d$. A gap $G$ of
$\lam_\sim$ is said to be \emph{critical} if $\si_d|_{G\,'}$ is at
least $k$-to-$1$ for some $k>1$. We define \emph{precritical} and
\emph{{\rm(}pre{\rm)}critical} objects similarly to how
(pre)periodic and preperiodic objects are defined above.
\end{defn}

For example, a periodic Siegel gap is non-critical even though
the first return map is not one-to-one on its basis (because there must be
critical leaves in the boundaries of gaps from its orbit).

\subsection{Geometric laminations}\label{ss:geol}

Laminations, understood as equivalence relations, can be described
in a geometric fashion, as was done in the original approach by
Thur\-ston \cite{thu85}. Thurston studied collections of chords in
$\disk$ similar to $\lam_\sim$, for a given $\si_d$-invariant
lamination $\sim$, with no lamination given.

\begin{defn}[Geometric laminations, \rm{cf.} \cite{thu85}]\label{geolam}
A \emph{geometric pre-lamination} $\lam$ is a set of (possibly
degenerate) chords in $\ol{\disk}$ such that any two distinct chords
from $\lam$ meet at most in a common endpoint; $\lam$ is called a
\emph{geometric lamination} (\emph{geo-lamination}) if all points of
$\uc$ are elements of $\lam$, and $\lam^+$ is closed. Elements of
$\lam$ are called \emph{leaves} of $\lam$. By a \emph{degenerate}
leaf (chord) we mean a singleton in $\uc$.
\end{defn}

In the Introduction (Definition \ref{d:gaps-i}), we defined gaps of
geo-laminations. Now let us discuss geo-laminations in the dynamical
context. Recall that given a chord $\ell=\ol{ab}$ of the unit disk
we define $\si_d(\ell)$ as the chord $\ol{\si_d(a)\si_d(b)}$ and
extend $\si_d$ linearly over $\ol{ab}$.

\begin{defn}[Invariant geo-laminations, \rm{cf.} \cite{thu85}]\label{geolaminv}
A geometric lamination $\lam$ is said to be an
\emph{$\si_d$-invariant} geo-lamination if the following conditions
are satisfied: 

\begin{enumerate}

\item (Leaf invariance) For each leaf $\ell\in \lam$, the set
    $\si_d(\ell)$ is a leaf in $\lam$ (if $\ell$ is critical, then $\si_d(\ell)$ is degenerate). For a
    non-degenerate leaf $\ell\in\lam$, there are $d$ pairwise disjoint
    leaves $\ell_1,\dots,\ell_d\in\lam$ with $\si_d(\ell_i)=\ell, 1\le i\le d$.

\item (Gap invariance) For a gap $G$ of $\lam$, the set
    $H=\ch(\si_d(G'))$ is a leaf, or a gap of
    $\lam$, in which case $\si_d:\bd(G)\to \bd(H)$
    is a positively oriented composition of a monotone map and a
    covering map (thus, if $G$ is a gap with finitely  many edges,
    all of which are critical, then its image  is a singleton).
\end{enumerate}

\end{defn}

Some invariant geo-laminations are not generated by laminations
(see, e.g., \cite{thu85}, where Thurston considers geo-laminations
with countable concatenations of leaves forming the boundary of a
gap). We will use a special extension $\si^*_{d, \lam}=\si_d^*$ of
$\si_d$ to the closed unit disk associated with $\lam$. On $\uc$ and
all leaves of $\lam$, we set $\si^*_d=\si_d$.
Define $\si^*_d$ on the interiors of gaps using a
standard barycentric construction \cite{thu85}. For brevity,
we sometimes use $\si_d$ instead of $\si^*_d$. We will mostly use
the map $\si_d^*$ if $\lam=\lam_\sim$ for some invariant lamination
$\sim$.

\subsection{Laminational sets and their basic properties}\label{ss:st}

So far we have dealt with (geo-)laminations. However, we also
consider subsets of $\ol{\disk}$ that have the properties of leaves
and gaps of geo-laminations while no actual geo-lamination is
specified. A number of facts can be proven for such sets, and we
establish some of them in this subsection.

\begin{defn}\label{d:restuff}
Let $f:X\to X$ be a self-mapping of a set $X$. For a set $G\subset
X$, let the \emph{return time} (to $G$) of $x\in G$ be the least
positive integer $n_x$ with $f^{n_x}(x)\in G$, or infinity if there
is no such integer. Set $n=\min_{x\in G} n_x$, define the set
$D_G=\{x\in G:n_x=n\}$, and call the map $f^n:D_G\to G$ the
\emph{remap} (first return map of $G$). Also, we define
\emph{refixed} points in $G$ as points $x\in G$ such that
$f^n(x)=x$. Similarly, we talk about 
\emph{reorbits} of points in $G$.
\end{defn}

For example, if $G$ is the boundary of a periodic Fatou domain of period
$n$ of a topological polynomial $f_\sim$, and the images $f^j_\sim(G),
j=0, 1, \dots,  n-1$ of $G$ are all pairwise disjoint until
$f^n_\sim(G)=G$, then $D_G=G$ and the remap on $D_G=G$ is
$f^n_\sim$.

By the \emph{relative interior} of a set in the plane, we mean the
interior of this set in its affine hull. Thus, the relative interior
of a gap of some lamination is its interior, while the relative
interior of a chord is the chord minus the endpoints.

\begin{defn}\label{d:lamset}
Throughout this definition we assume that $A\subset\uc$ is closed.
If all the sets $\ch(\si_d^i(A))$ are pairwise disjoint, then $A$ is
called \emph{wandering}. If there exists $n\ge 1$ such that all the
sets $\ch(\si_d^i(A)), i=0, \dots, n-1$ have pairwise disjoint
relative interiors while $\si_d^n(A)=A$, then $A$ is called
\emph{periodic} of period $n$. If there exists $m>0$ such that all
$\ch(\si_d^i(A)), 0\le i\le m+n-1$ have pairwise disjoint relative
interiors and $\si_d^m(A)$ is periodic of period $n$, then we call
$A$ \emph{preperiodic}. Observe that the above applies to sets $A$
regardless of whether they are a part of a (geo-)lamination or not.

If $A$ is wandering, periodic or preperiodic, and for every $i\ge 0$
and every hole $(a, b)$ of $\si_d^i(A)$ either $\si_d(a)=\si_d(b)$,
or the positively oriented arc $(\si_d(a), \si_d(b))$ is a hole of
$\si_d^{i+1}(A)$, then we call $A$ (and $\ch(A)$) a
\emph{{\rm(}$\si_d$-{\rm)}laminational set}; we call both $A$ and
$\ch(A)$ \emph{finite} if $A$ is finite. A {\em
{\rm(}$\si_d$-{\rm)}stand alone gap $G$} is defined as a
laminational set with non-empty interior (note that a gap of a
geo-lamination always has non-empty interior). In other words, a
stand alone $G$ is of the form $\ch(A)$ for some closed set
$A\subset \uc$ of more than two points such that the above listed
properties hold for $A$.
\end{defn}

In what follows, whenever we say that $G$ is a ``gap'' we mean that
at least $G$ is a stand alone gap (it will be clear from the context
if there is a lamination or a geo-lamination such that $G$ is a part
of it). Also, abusing the language, we will sometimes identify closed
sets $A\subset \uc$ and their convex hulls (again, it will be clear
from the context what kind of set we consider).

The basis $G'=G\cap \uc$ of a gap $G$ coincides
with the union $A\cup B$ of two well-defined sets, where $A$ is a
maximal Cantor subset of $G'$ or an empty set and $B$ is countable.
Assume that $A\ne \0$, and define a map $\psi_G:\uc\to \uc$ that
collapses all holes of $A$ to points.
Suppose that $G$ is $m$-periodic.
It is well-known that $\psi_G$ can be chosen so that it semiconjugates $\si^m_d|_{G'}$ to an irrational rotation of the circle or to the map $\si_k$, where $k\ge 2$.
Indeed, by Definition~\ref{d:lamset}, the map of the circle, to which
$\si_d$ is semi-conjugate under $\psi_G$, is locally 1-to-1 and orientation
preserving. On the other hand, the fact that $\si_d$ is locally
expanding implies that the induced map of the circle does not have
wandering arcs (see, e.g., \cite{blolev02a} where similar arguments
were used in classifying different types of gaps of laminations).
This implies the above claim.

Accordingly, if $A\ne \0$, we call a stand alone periodic gap $G$ a
\emph{stand alone periodic Fatou gap of degree $k$} if in the above
construction the map $\psi_G$ semiconjugates $\si^m_d|_{G'}$ to
$\si_k, k\ge 2$. Also, we call a stand alone periodic gap $G$ a
\emph{stand alone periodic Siegel gap} if, in the above construction,
the map $\psi_G$ semiconjugates $\si^m_d|_{G'}$ to an irrational
rotation. Moreover, for periodic laminational sets $G$ with finite
basis $G'$ and for periodic stand alone Siegel gaps $G$ we can
define the \emph{rotation number} $\tau_G$. If the rotation number
is not equal to zero, the set $G$ is said to be \emph{rotational}.
If such $G$ is invariant, we call it an \emph{invariant rotational
set}.

\begin{lem}[Lemma 2.16 \cite{bopt11}]\label{l:maj}
Suppose that $\ell=\ol{xy}$ is a laminational chord such that there
exists a component $Q$ of the complement of its orbit in the disk
$\disk$ whose closure contains $\si_d^n(\ell)$ for all $n\ge 0$.
Then the chord $\ell$ is either {\rm(}pre{\rm)}critical or
{\rm(}pre{\rm)}periodic.
\end{lem}

Let $U$ be the convex hull of a closed subset $A$ of $\uc$. For
every edge $\ell$ of $U$, let $H_U(\ell)$ denote the hole of $U$
whose endpoints coincide with the endpoints of $\ell$. In this
situation, we define $|\ell|_U$ as $|H_U(\ell)|$. Notice that if $U$
is a chord, then it has two holes (on opposite sides of $U$).

Suppose in addition that $A\subset \uc$ is a laminational set.
Mostly, the holes of $A$ map increasingly onto the holes of
$\si_d(A)$. However, if the length of a hole is at least
$\frac1{d}$, then the map $\si_d$ wraps the hole around the circle
one or more times. Thus, holes $H$ of $U$ such that $|H|\ge \frac1d$
are the only holes which the map $\si_d$ does not take one-to-one
onto their images.

\begin{defn}\label{d:0major}
Let $G$ be the convex hull of a closed subset of $\uc$. If $\ell$ is
an edge of $G$ such that its hole $H_G(\ell)$ is not shorter than
$\frac 1d$, then $\ell$ is called a ($\si_d$-){\em major edge of
$G$} (or simply a ($\si_d$-){\em major of $G$}), and $H_G(\ell)$ is
called a ($\si_d$-){\em major hole of $G$}.
\end{defn}

It is useful to work with a wider class of sets, which we now
introduce.

\begin{defn}\label{d:semila}
A closed set $A\subset\uc$ (and its convex hull) is said to be
($\si_d$-){\em semi-laminational} if, for every hole $(x,y)$ of $A$,
we have $\si_d(x)=\si_d(y)$, or the open arc $(\si_d(x),\si_d(y))$
is a hole of $A$ (the set $A$ is not assumed to satisfy
$\si_d(A)=A$).
\end{defn}

A set is said to be \emph{invariant} if it maps into itself.
Clearly, a $\si_d$-invariant laminational set is
$\si_d$-semi-laminational.

For a chord $\ell$, let $\orb_{\si_d}(\ell)$ denote the union of all
chords in the forward orbit of $\ell$ under $\si_d$.

\begin{lem}\label{l:lameqsemi} If $\ell=\ol{xy}$ is a
$\si_2$-periodic chord of period $r$ with $\si_2^r$-fixed endpoints
and if there is a unique component $Z$ of $\ol{\disk}\sm
\orb_{\si_2}(\ell)$ such that all images of $\ell$ are edges of $Z$,
then $\ol{Z}$ is a finite $\si_2$-invariant stand alone gap.
\end{lem}

\begin{proof}
We begin by making an observation which applies to all maps $\si_d$
(the situation of the lemma can be described for $\si_d$ instead of
$\si_2$). If $I=(\si_d^k(x), \si_d^k(y))$ is a hole of $\ol Z$ while
$(\si_d^{k+1}(y), \si_d^{k+1}(x))$ is a hole of $\ol Z$, then we say
that $\si_d$ \emph{changes orientation} on $I$. Let us show that
if $I=(\si_d^k(x), \si_d^k(y))$ is a hole of $\ol Z$ such that
$\si_d$ changes orientation on $I$, then $I$ contains a
$\si_d$-fixed point. Indeed, the image of $I$ is an arc which
connects $\si_d^{k+1}(x)$ to $\si_d^{k+1}(y)$ and potentially wraps
around the circle a few (zero or more) number of times. If
$(\si_d^{k+1}(y), \si_d^{k+1}(x))$ is a hole of $\ol Z$, it follows that $I\subset \si_d(I)$ and that the endpoints of $\si_d(I)$ do not
belong to $I$. Hence $I$ contains a $\si_d$-fixed point.

Let us now prove the lemma. By the above, $\si_2$ changes
orientation on at most one hole. Since $\ell$ is periodic and its
endpoints are refixed, $\si_2$ changes orientation an even number of
times. Hence $\si_2$ never changes orientation on holes of $\ol Z$.
This implies that $Z$ is semi-laminational. Now, if a hole of $\ol Z$ is shorter than $\frac12$, then it doubles in length under $\si_2$
while still being mapped onto its image one-to-one. Hence there
exists a hole $H$ of $Z$ whose length is at least $\frac12$. If we
draw a critical leaf $c$ with endpoints in $H$, we see that the entire
orbit of $\ell$ consists of leaves with endpoints in the complement of $H$.

Now, it is well known (see, e.g., \cite{thu85}) that a
periodic orbit of $\si_2$ contained in a given half-circle  (in our
case, this is any half-circle containing $\uc\sm H$)
is the point $0$, or the pair $\{1/3, 2/3\}$, or the set of
vertices of a finite invariant stand alone $\si_2$-gap.
Moreover, a given half-circle contains exactly one periodic $\si_2$-orbit.
It follows that the endpoints of $\ell$ are vertices of an invariant $\si_2$-gap $\ol Z$. Clearly, $\ell$ then
has to be an edge of $\ol Z$ (if $\ell$ is a diagonal of $\ol Z$,
then distinct images of $\ell$ intersect inside
$\disk$).
\end{proof}

The class of semi-laminational sets is wider than the
class of invariant laminational sets because
Definition~\ref{d:semila} allows for circle arcs to be parts of
semi-laminational sets.
For example, take a $\si_d$-invariant stand alone Fatou gap $G$
of degree $k>1$ such that there is a periodic orbit $Q$ of edges of
$G$. Let $H_1, \dots, H_n$ be the holes of $G$ behind edges from
$Q$. Then $A=\uc\sm\bigcup^n_{i=1}H_i$ is a semi-laminational set.
The assumption that a Fatou gap like $G$ exists means that $d$ must
be greater than $2$. In Lemma~\ref{l:posholes} we study
semi-laminational sets for cubic laminations.

Majors of semi-laminational sets
play an important role because, as we see below in Lemma~\ref{l:fx-maj}, all
edges of semi-laminational sets have majors of these semi-laminational sets in their
forward orbits.

\begin{lem}
\label{l:fx-maj} An edge of a semi-laminational set $G=\ch(A)$
$($where $A\subset \uc$ is closed$)$ is a major
if and only if the closure of its hole contains a fixed point. Any
edge of $G$ eventually maps to a major of $G$.
\end{lem}

\begin{proof}
Let $\ell$ be an edge of $G$. The case when $\ell$ is invariant (i.e. such that $\si_d(\ell)=\ell$) is
left to the reader. Otherwise if $|H_G(\ell)|<1/d$, then $H_G(\ell)$
maps onto the hole $\si_d(H_G(\ell))$ one-to-one. The fact that
$\ell$ is not invariant implies that $\si_d(H_G(\ell))$ is disjoint
from $H_G(\ell)$. Hence $H_G(\ell)$ contains no fixed points. On the
other hand, suppose that $|H_G(\ell)|\ge 1/d$. Then
$\si_d(H_G(\ell))$ covers the entire $\uc$ while the images of the
endpoints of $H_G(\ell)$ are outside $H_G(\ell)$. This implies that
there exists a fixed point $a\in \ol{H_G(\ell)}$. To prove the
second claim, choose an edge $\ell$ of $G$. For any $i$ set
$T_i=H_G(\si^i_d(\ell))$. As long as $|T_i|<\frac{1}d$, we have
$|T_{i+1}|=|\si_d(T_i)|=d|T_i|$. Hence there exists the least $n$
such that $|T_n|\ge \frac{1}d$. Then the leaf $\si_d^n(\ell)$ is a
major of $G$, as desired.
\end{proof}


\subsection{Fixed points and invariant sets}

Theorem~\ref{t:fxpt} allows one to find specific invariant sets in
some parts of the disk. We state it in the language of laminations.

\begin{thm}[\cite{bfmot10}]\label{t:fxpt}
Let $\sim$ be a $\si_d$-invariant lamination. Suppose that $e_1,
\dots, e_m$ are some leaves of $\sim$ and $X$ is a component of $\disk\sm
\bigcup^m_{i=1} e_i$ such that for each $i$
\begin{enumerate}
\item the leaf $e_i$ lies on the
boundary of $X$,
\item there exists no finite gap of $\sim$ inside $X$ with
an edge $e_i$, and
\item either $\si_d$ fixes each endpoint of
$e_i$, or $\si_d(e_i)$ is contained in the component of
$\ol{\disk}\sm e_i$ that contains $X$.
\end{enumerate}
Then at least one of the
following claims holds:

\begin{enumerate}

\item $X$ contains an invariant gap of $\sim$ of degree $k>1$;

\item $X$ contains an invariant rotational set.

\end{enumerate}

\end{thm}

\section{Invariant quadratic gaps and their canonical laminations}\label{s:qgaps}

By cubic laminations we mean $\si_3$-invariant laminations. By a
\emph{quadratic} gap we mean a stand alone periodic Fatou gap $U$ of
degree $2$. In this section, we assume that $U$ is a
$\si_3$-invariant quadratic gap and study its properties. We then
define \emph{canonical} laminations, which correspond to these gaps
and describe other laminations that \emph{refine} the canonical
ones. Throughout the rest of the paper we will often write $\si$
instead of $\si_3$.

\subsection{Invariant quadratic gaps}\label{s:invquagap}

Recall that, given a gap $U$ with an edge $\ell$, we write
$|\ell|_U$ for $|H_U(\ell)|$. If $U$ is given, we may drop the
subscript $U$ from the notation.

\begin{lem}
\label{bndcrit}
Let $U$ be a $\si$-invariant stand alone quadratic gap.
Then there exists a unique major edge $\ell$ of $U$, on all holes $\widetilde H\ne H(\ell)$ of $U$ the map $\si$ is a homeomorphism onto its image, $|\si(\widetilde H)|=3|\widetilde H|$, and the following cases are possible:

\begin{enumerate}
\item we have $|\ell|_U=\frac13$, the leaf $\ell$ is not periodic, and all holes
$\widetilde H\ne H(\ell)$ of $U$ are of length at most $\frac19$;

\item the leaf $\ell$ is periodic of some period $k$, we have $\frac13<|\ell|_U\le \frac12$, and $|\ell|_U=\frac12$ only if $\ell=\ol{0\frac12}$.
\end{enumerate}
\end{lem}

\begin{proof}
The existence of a major $\ell$ follows from Lemma \ref{l:fx-maj}.
Observe that, if a set $A\subset\uc$ lies
in the complement of two disjoint closed arcs in $\uc$ of length $\ge 1/3$
each, then the restriction of $\si$ to $A$ is injective. This
implies that all holes $\widetilde H\ne H(\ell)$ of $U$ are shorter than
$\frac13$ and that $|\ell|_U<\frac23$ (here we use the fact that
$\si|_{U'}$ is two-to-one).

Clearly, $|\ell|_U$ can be equal to $\frac13$ (just take
$\ell=\ol{\frac13 \frac23}$ and assume that $U$ is the convex hull of
the set of all points $x\in \uc$ with orbits outside the arc $(\frac13,
\frac23)$). This situation corresponds to case (1) of the lemma. Since $\si$ expands
the length by the factor of $3$, it follows that all holes $\widetilde H\ne H(\ell)$ of $U$ are shorter than
$\frac19$.

Suppose that $|\ell|_U>\frac13$. Then $\si(\ell)$ is eventually
mapped to $\ell$ by Lemma \ref{l:fx-maj}, hence $\ell$ is periodic.
Since $\frac13<|\ell|_U=x<\frac23$, then $|\si(\ell)|_U=3x-1$. If
the period of $\ell$ is $k$, then $3^{k-1}(3x-1)=x$ and so
$x=|\ell|_U=3^{k-1}(3^k-1)^{-1}\le \frac12$ with equality possible
only if $k=1$ in which case clearly $\ell=\ol{0\frac12}$
(observe that under the assumptions of the lemma the leaf $\ell$
maps onto itself by $\si_3^k$ so that each endpoint of $\ell$ maps
to itself). This corresponds to case (2) of the lemma.
\end{proof}

\begin{lem}
  \label{l:bndcrit1}
In case $(2)$ of Lemma \ref{bndcrit}, there exists a unique leaf $\ell^*$ disjoint from $\ell$ with endpoints in $H(\ell)$ such that $\si(\ell^*)=\si(\ell)$ and such that one of the following holds:
\begin{enumerate}
\item we have $\ell=\ol{0\frac12}$, the only possible holes of $U\cup \ell^*$ of length $\frac16$ are
$$
\left(0, \frac16\right),
\left(\frac16, \frac13\right),
\left(\frac13, \frac12\right),
\left(\frac12, \frac23\right),
\left(\frac23, \frac56\right),
\left(\frac56, 1\right),
$$
and all other holes of $U\cup \ell^*$ are of length at most $\frac1{18}$;
\item we have $\ell\ne \ol{0\frac12}$, every hole of $U\cup \ell^*$ is shorter than $\frac16$ except for the hole of $\ell^*$ that is disjoint from $U$ and has length greater than $\frac5{18}$.
\end{enumerate}
\end{lem}

\begin{proof}
Let $\ell=\ol{ab}$ and observe first that $\ell^*=\ol{b^*a^*}$, where $b^*=b-\frac13, a^*=a+\frac13$.
Consider now two cases.
First, assume that $\ell=\ol{0\frac12}$.
Then it is easy to see that the claims of the lemma hold.
Since $\ol{0\frac12}$ is a unique major edge $\ell$ of an invariant quadratic gap with $|\ell|=\frac12$, from now we may assume that $\ell<\frac12$.

Consider a hole $\widetilde H$ of $U\cup \ell^*$.
If $\widetilde H$ is a hole of $U$, then it maps forward monotonically (and hence expanding by the factor of $3$) a few times before it maps onto $H(\ell)$.
Since, by the above, $|H(\ell)|\le \frac12$, we have $|\widetilde H|<\frac16$.
Now, $|(a, b^*)|=|H(\ell)| - |\frac13<\frac12-\frac13=\frac16$, and, similarly, $|(a^*, b)|<\frac16$.
Finally, consider the arc $(b^*,a^*)$.
Its image is the arc complementary to the arc $(\si(a),\si(b))$.
Since we assume that $\ell\ne \ol{0\frac12}$, it follows that $(\si(a),\si(b))$ is a hole of $U$ distinct from $H(\ell)$.
Hence the length of $(\si(a), \si(b))$ is less than $\frac16$, which implies that
the length of $(\si(b), \si(a))$ is more than $\frac56$, and hence the length of
$(b^*, a^*)$ is more than $\frac5{18}$.
\end{proof}

From now on, the major $\ell$ of $U$ is denoted by $M(U)$.
Lemma~\ref{bndcrit} implies a simple description of the basis $U'$ of $U$ if $M(U)$ is given.

\begin{lem}\label{descru}
The basis $U'$ of $U$ is the set of all points $x\in \uc$, whose
orbits are disjoint from $H(M(U))$. All edges of $U$ are preimages
of $M(U)$.
\end{lem}

\begin{proof}
Orbits of all points of $U'$ are disjoint from $H(M(U))$. Also, if
$x\in \uc\sm U'$, then $x$ lies in a hole of $U$ behind a leaf
$\ell$. By Lemmas~\ref{l:fx-maj} and~\ref{bndcrit}, the orbit of
$\ell$ contains $M(U)$. Hence for some $n$ we have
$\si^n(\ell)=M(U)$, $\si^n(H(\ell))=H(M(U))$, and $\si^n(x)\in
H(M(U))$.
\end{proof}

It is natural to consider the two cases from Lemma~\ref{bndcrit} separately.
We begin with the case, when an invariant quadratic gap $U$ has a periodic major $M(U)=\ol{ab}$ of period $k$ and $H(M(U))=(a, b)$.
Set $a^*=a+\frac13$ and $b^*=b-\frac13$.
Set $M^*(U)=\ol{b^*a^*}$.
Consider the set $N(U)$ of all points of $\uc$ with $\si^k$-orbits contained in $[a, b^*]\cup [a^*, b]$ and its convex hull $V(U)=\ch(N(U))$ (this notation is used in several
lemmas below).
We call $V(U)$ the \emph{vassal} (gap) of $U$ (see Figure 1).

\begin{lem}\label{vassal}
Assume that $M(U)$ is periodic of period $k$.
Then $N(U)$ is a Cantor set; $V(U)$ is a periodic quadratic gap of period $k$.
\end{lem}

\begin{proof}
Under the action of $\si$ on $[a, b]$ the arc $[a,a^*]$ wraps around the circle once, and the arc $[a^*, b]$ maps onto the arc $[\si(a),\si(b)]$ homeomorphically.
Similarly, we can think of $\si|_{[a, b]}$ as homeomorphically mapping $[a, b^*]$ onto the arc $[\si(a),\si(b)]$ and wrapping $[b^*,b]$ around the circle once.
Thus, first the arcs $[a,b^*]$ and $[a^*, b]$ map homeomorphically to the arc $[\si(a), \si(b)]$ (which, by definition, is the closure of a hole of $U$).
Then, under further iterations of $\si$, the arc $[\si(a),\si(b)]$ maps homeomorphically onto closures of distinct holes of $U$ until $\si^{k-1}$ sends it, homeomorphically, onto $[a,b]$ (all this follows from Lemma~\ref{bndcrit}).
This generates the quadratic gap $V(U)$ contained in the strip between $\ol{ab}$ and
$\ol{a^*b^*}$.
In the language of one-dimensional dynamics (see \cite{mis79, ms80, mt88, you81} and the book \cite{alm00}) one can say that closed intervals $[a,b^*]$ and $[a^*, b]$ form a
2-horseshoe of period $k$.
A standard argument shows that $N(U)$ is a Cantor set.
The remaining claim easily follows.
\end{proof}

\begin{figure}[htp]
\begin{center}
\includegraphics[width=6cm]{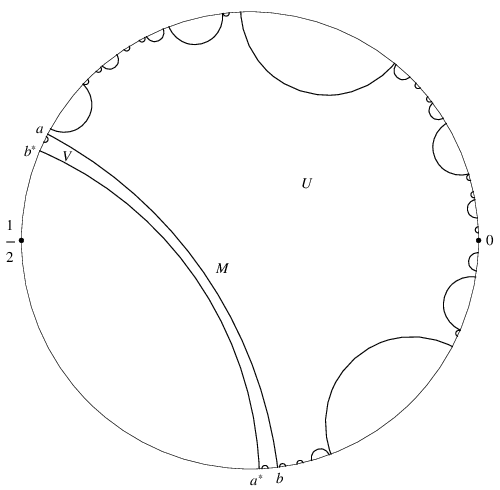}
{\caption{A quadratic invariant gap $U$ of periodic type and its
vassal $V$. We draw geodesics in the Poincar\'e metric instead of
straight chords to make the pictures look better.}}
\end{center}
\end{figure}

We also define another type of gap called a \emph{caterpillar}
(gap). This is a periodic gap $Q$ with the following properties:
\begin{itemize}
 \item  The boundary of $Q$ consists of a periodic (possibly degenerate)
 leaf $\ell_0=\ol{xy}$ of period $k$ called a \emph{head of the caterpillar gap $Q$} ,
a $\si^k$-critical leaf $\ell_{-1}=\ol{yz}$ concatenated to it, and
a countable concatenation of leaves $\ell_{-n}$ accumulating at $x$
(the leaf $\ell_{-r-1}$ is concatenated to the leaf $\ell_{-r}$, for
every $r=1$, 2, $\dots$).
\item
We have $\si^k(x)=x$, $\si^k(\{y, z\})=\{y\}$, and $\si^k$ maps
each $\ell_{-r-1}$ to $\ell_{-r}$ (all leaves are shifted by one
towards $\ell_0$ except for $\ell_0$, which maps to itself, and
$\ell_{-1}$, which collapses to the point $y$).
\end{itemize}
Similar gaps are already useful for quadratic laminations (see
\cite{thu85} where the invariant gap with edges $\ol{0\frac 12},
\ol{\frac 12\frac 14}, \dots, \ol{\frac 1{2^n}\frac 1{2^{n+1}}},
\dots$ is considered). Lemma~\ref{catpil} is left to the reader.


\begin{lem}\label{catpil}
Suppose that $M(U)=\ol{ab}$ is periodic of period $k$. Then one can construct
two caterpillar gaps with head $M(U)$ such that their bases are contained
in $\ol{H(M(U))}$. In the first of them, the critical leaf is
$\ol{aa^*}$, and in the second one the critical leaf is $\ol{bb^*}$.
\end{lem}

\begin{proof}
Let us describe one of the caterpillar gaps in question.
Since the the hole $(a, b)$ of $U$ behind $M(U)$ has length between $\frac13$ and $\frac23$, there exists a unique point $a^*\in (a, b)$ such that $\si_3(a^*)=\si_3(a)$ and there exists a unique point $b^*\in (a,b)$ such that $\si_3(b^*)=\si_3(b)$.
Let $\ell_0=M(U)$ and $\ell_{-1}=\ol{aa^*}$ (clearly, $\ol{aa^*}$ is a $\si_3$-critical
leaf).
Choose the unique $\si_3^k$-preimage $a_{-2}$ of $a_{-1}=a^*$ such
that $a<a_{-1}<a_{-2}<b$. Put $\ell_{-2}=\ol{a_{-1}a_{-2}}$.
Consider further pullbacks $\ell_{-n}$ of $\ell_0$ so that they form a concatenation
of leaves converging to the point $b$.
Then the required $\si_3^k$-fixed caterpillar gap is the convex hull of the union of
leaves $\bigcup^\infty_{i=0}\ell_{-i}$.
The other desired caterpillar gap with critical leaf $\ol{bb^*}$ can be constructed
similarly.
\end{proof}

We call the caterpillar gaps from Lemma~\ref{catpil} \emph{canonical caterpillar gaps of $U$}.
A critical edge $c$ of a canonical caterpillar gap defines it, so this caterpillar gap is denoted by $C(c)$.
We denote its basis by $C'(c)$. To study related invariant quadratic gaps we first
prove the following general lemma, in which we adopt a different
point of view. Namely, any leaf $\ell$ that is not a diameter
defines an open arc $L(\ell)$ (a component of $\uc\sm \{\ell\}$) of
length greater than $\frac12$ (in particular, a critical leaf $c$
defines an arc $L(c)$ of length $\frac23$). Let $\Pi(\ell)$ be the
set of all points with orbits in $\ol{L(\ell)}$.

\begin{lem}\label{pic}
Suppose that $c$ is a critical leaf. The set $\Pi(c)$ is non-empty,
closed and forward invariant. A point $x\in \Pi(c)$ has two
preimages in $\Pi(c)$ if $x\ne \si(c)$, and three preimages in
$\Pi(c)$ if $x=\si(c)$. The convex hull $G(c)$ of $\Pi(c)$ is a
stand alone invariant quadratic gap. If $\si(c)\in \Pi(c)$, we
have the situation of case {\rm(1)} of Lemma~\ref{bndcrit} and $c$
is the major of $G(c)$.
\end{lem}

Observe that, since $c$ is a critical leaf, its $\si$-image is a point.

\begin{proof}
It is easy to see that the set $\Pi(c)$ is closed and forward invariant;
it is non-empty because it contains at least one fixed point (indeed, as was noticed before
the lemma, the circle arc $L(c)$  in the case at hand is of length $\frac23$, and each circle arc of length
$\frac23$ contains at least one $\si_3$-fixed point). Let $x\in
\Pi(c)$. If $x\ne \si(c)$, then, of its three preimages, one belongs to
$\uc\sm \ol{L(c)}$ while two others are in $\ol{L(c)}$, and hence, by definition,
also in $\Pi(c)$. Suppose that $x=\si(c)$ (and so, since by the
assumption $x\in \Pi(c)$, the orbit of $c$ is contained in $\ol{L(c)}$).
Then the entire triple $\si^{-1}(\si(c))$ is contained in $\ol{L(c)}$ and,
again by definition, in $\Pi(c)$.

To prove the next claim of the lemma, we prove that any hole $I$ of
$\Pi(c)$ except for the hole $T$, whose closure contains the
endpoints of $c$, maps to a hole of $\Pi(c)$. Indeed, otherwise
there is a point $y\in I$ such that $\si(y)$ is a point of $\Pi(c)$.
Since $I\subset \ol{L(c)}$, we have $y\in \Pi(c)$, a contradiction.
Consider the hole $T=(a,b)$ defined above. If $T=\uc\sm \ol{L(c)}$,
there is nothing to prove as in this case the leaf $c$ is a critical
edge of $G(c)$ that maps to a point of $\Pi(c)$. Suppose now that
$T\ne \uc\sm \ol{L(c)}$ and that there is a point $z\in (\si(a),
\si(b))\cap \Pi(c)$. Then there is a point $t\in T\cap L(c)$ with
$\si(t)=z$ and hence $t\in \Pi(c)$, a contradiction. Thus, $(\si(a),
\si(b))$ is a hole of $\Pi(c)$, and $\ol{\si(a)\si(b)}$ is an edge
of $G(c)$. All this implies that $G(c)$ is an invariant gap, and it
follows from the definition that it is quadratic. The last claim
easily follows.
\end{proof}

Below we will use the notation $G(c)=\ch(\Pi(c))$. Let us relate
invariant quadratic gaps defined in terms of periodic majors (see
Lemma~\ref{vassal}) and caterpillar gaps $C(c)$ (see
Lemma~\ref{catpil}) to the gaps $G(c)$. To state Lemma~\ref{picper},
we need the following concept. By Lemmas~\ref{bndcrit} and
\ref{descru}, all holes of $U$ map onto $H(M(U))$ after finitely many
steps in a monotone fashion. Suppose that $I$ is a hole of $U$ and
$n\ge 0$ is a positive integer such that $\si^n(I)=H(M(U))$ and
$\si^n|_I$ is monotone. Then for any point $x\in H(M(U))$ we call
a unique point $\widetilde x\in I$ with $\si^n(\widetilde x)=x$ the \emph{first
pullback of $x$ to $I$}.

\begin{lem}\label{picper}
Let $M(U)=\ol{ab}$ be periodic of period $k$ and $c=\ol{xy}$ be a
critical chord with endpoints in $\ol{H(M(U))}$. Then:
\begin{enumerate}
\item if $x$, $y\in H(M(U))$, then $U'=\Pi(c)$,
\item otherwise $\Pi(c)$ consists of $U'$, $C'(c)$, and all first pullbacks of
points of $C'(c)$ to holes of $U$.
\end{enumerate}
\end{lem}

\begin{proof}
Clearly $U'\subset \Pi(c)$.
We may assume that the circle arc $[a,b]$ is ordered in the positive direction from $a$ to $b$.
Let $\ell=\ol{b^*a^*}$ be the non-periodic edge of $V(U)$ such that $\si(a^*)=\si(a)$, $\si(b^*)=\si(b)$; suppose that $\ell\cap c=\0$.
The order of points on the circle arc $[a, b]$ is as follows: $a<b^*<a^*<b$.
By construction of $V(U)$, we have that $I=H_{V(U)}(\ell)$ coincides with $(b^*,a^*)$,
which in particular implies that the length of $I$ is less than $\frac13$ (because the length of $(a,a^*)$ is $\frac13$).
Moreover, $\ol I$ contains no points of $\Pi(c)$.
Since $c$ is disjoint from $\ell$ and has endpoints in $\ol{H(M(U))}$, we may assume that $x\in (a, b^*)$ and $y\in (a^*,b)$.
Thus, we have $J=(x, y)\supset I$.
Note that the restriction of $\si^k$ to $H(M(U))\sm \ol{J}$ is a one-to-one expanding map to $H(M(U))$.
It follows that all points of $H(M(U))$ eventually map to $\overline{J}$, therefore, they do not belong to $\Pi(c)$ either.
By Lemma~\ref{descru}, any point $x\nin U'$ eventually maps to $H(M(U))$.
Thus, $\Pi(c)=U'$.
The case, where $c=\ol{aa^*}$ or $c=\ol{b^*b}$, is left to the reader.
\end{proof}

Lemma~\ref{l:perpic} complements Lemma~\ref{picper}.

\begin{lem}\label{l:perpic}
If $c=\ol{xy}$ is a critical leaf with a periodic endpoint, say, $y$
such that $x$, $y\in \Pi(c)$ then there exists a quadratic invariant
gap $W$ with a periodic major $M(W)=\ol{zy}$ of period $k$, the point $z$ is
the closest to $x$ in $L(c)=(y, x)$ point that is $\si^k$-fixed, and
$z$ is of period $k$.
\end{lem}

\begin{proof}
Assume that the period of $y$ is $k$. Consider the gap $G(c)$. By
the invariance of $G(c)$,
there is an edge $c_{-1}$ of
$G(c)$ attached to $c$ at $x$ which maps to $c$ under $\si^k$.
Moreover, $|c_{-1}|_{G(c)}=3^{-k-1}$. This can be continued
infinitely many times so that the $m$-th edge of $G(c)$, which maps
to $c$ under $\si^{mk}$, is a leaf $c_{-m}$ such that its hole is of
length $3^{-mk-1}$. Clearly, the concatenation $A$ of leaves $c$,
$c_{-1}$, $\dots$ converges to a point $z\in \uc$ which is
$\si^k$-fixed. Set $\ol{zy}=M$.

Since, by Lemma~\ref{pic}, the gap $G(c)$ is quadratic, there are
many preimages of $\ol A$ on the boundary of $G(c)$. Replace them all by
the corresponding preimages of $M$ (e.g., replace $\ol{A}$ by $M$).
It follows that the newly constructed gap $W$ is  a quadratic invariant stand alone Fatou gap
with  major $M=M(W)$
as desired. The last claim of the lemma easily follows.
\end{proof}

Let us summarize the above results. Let $c$ be a critical leaf. We
will now define an invariant stand alone quadratic Fatou gap $U(c)$
determined by $c$. Even though in the beginning of this section we
announced that we will consider a given  invariant stand alone
quadratic Fatou gap $U$, we choose the notation $U(c)$ for the gap
determined by $c$  to reflect the fact that $U(c)$ has the same
properties (i.e., that it is an invariant stand alone quadratic
Fatou gap). In what follows we will use the notation $L(c), \Pi(c)$
and $G(c)$ introduced above.

The orbit of a critical leaf $c$ can be of three types. First, the
orbit of $c$ can be contained in $L(c)$ so that no endpoint of $c$
is periodic. Then we set $U(c)=G(c)$ and call $U(c)$ and the leaf
$c$ \emph{regular critical}. Second, an endpoint of $c$ can be
periodic with the orbit of $c$ contained in $\ol{L(c)}$. By
Lemmas~\ref{picper} and ~\ref{l:perpic}, then $\Pi(c)$ consists of
$U'$, $C'(c)$, and first pullbacks of $C'(c)$ to holes of $U$ for
some invariant quadratic gap $U=U(c)$ with a periodic major $M(U)$
(the gap $U$ can be defined as the convex hull of all non-isolated
points in $\Pi(c)$). This defines the gap $U(c)$ in that case. Then
we call the gap $G(c)$ an \emph{extended \ca\, gap}, and the
critical leaf $c$ a \emph{\ca\,} critical leaf. Third, there can be
$n>0$ with $\si^n(c)\nin \ol{L(c)}$. Denote by $n_c$ the smallest
such $n$. Then $G(c)$ has a periodic major of period $n_c$ and we
set $U(c)=G(c)$ and call this gap a \emph{gap of periodic type}.
Only regular critical gaps or gaps of periodic type can be invariant
quadratic gaps of laminations (a critical leaf with a periodic
endpoint would imply that the corresponding infinite canonical
caterpillar is contained in one class, a contradiction with
condition (E3) of Definition~\ref{d:lam}).

We show below in Lemma~\ref{l:posholes} that properties listed in item (2) of
Lemma~\ref{bndcrit} are basically sufficient for a periodic leaf
$\ell$ to be the major of a quadratic invariant gap of periodic
type. The arguments in Lemma~\ref{l:posholes} are related to those
in Lemma~\ref{l:lameqsemi}.

\begin{lem}\label{l:posholes}
Let $\ell=\ol{xy}$ be a $\si$-periodic leaf of period $r$ with
$\si^r$-fixed endpoints for which there is a unique component $Z$ of
$\ol{\disk}\sm \orb_\si(\ell)$ such that any two iterated images of $\ell$ are
disjoint or coincident edges of $Z$. Then a hole of $Z$ with length
greater than $\frac13$ exists if and only if $Z$ is a
semi-laminational set. In that case an eventual $\si$-image
of $\ell$ that corresponds to the major hole of $Z$
is a major of a quadratic invariant gap of periodic type.
\end{lem}

Observe that $Z$ has $r$ holes each of which is located behind an
image of $\ell$ and has length equal to neither $\frac13$ nor
$\frac23$. Recall that if a Fatou gap $G$ is invariant, then the
quotient map $\psi_G:\bd(G)\to\uc$ is defined as a map collapsing
all edges of $G$ to points and mapping $G$ to the unit circle.

\begin{proof}
First assume that $Z$ is semi-laminational.
Then, by the above remark and Lemma~\ref{l:fx-maj}, at least one hole $H$ of $Z$ must be longer than $\frac13$.
We may assume that $H=H_Z(\ell)$. Choose a critical leaf $c$ whose non-periodic endpoints are in $(x, y)$. Let us show that $G(c)$ is of periodic type and $\ell$ coincides with the major $M$ of the gap $G(c)$. Indeed, suppose otherwise. By definition the endpoints of $\ell$ belong to $G'(c)$ and $\psi_{G(c)}$ maps $\ell$ to a leaf
$\psi_{G(c)}(\ell)$ such that the leaf $\psi_{G(c)}(\ell)$ and its $\si_2$-images satisfy conditions of Lemma~\ref{l:lameqsemi}. Hence
$\psi_{G(c)}(\ell)$ and its $\si_2$-images are the edges of a finite
$\si_2$-invariant gap. In particular, they are not pairwise disjoint.

Now, $\psi_{G(c)}$ collapses only preimages of $M$. If $G(c)$ is of
regular critical type (i.e., $M=c$ has no periodic endpoints), it
will follow that $\ell$ and its $\si$-images are not pairwise
disjoint, a contradiction. If $G(c)$ is of periodic type (i.e., $M$
is a periodic leaf) then, if $\ell$ and its $\si$-images miss
endpoints of $M$, we have that $\psi_{G(c)}$ is one-to-one on the
endpoints of $\ell$ and its $\si$-images. Hence $\ell$ and its
$\si$-images are not pairwise disjoint, a contradiction.

Suppose finally that, say, $\ell\ne M$ shares an endpoint $x$ with
$M=\ol{xz}, z\ne y$. Since $\ell$ and its $\si$-images are pairwise
disjoint, $y$ does not belong to the same periodic orbit as $x$. On
the other hand, $\psi_{G(c)}$-images of leaves from the $\si$-orbit
of $\ell$ are the edges of a finite $\si_2$-invariant gap. Thus, $y$
belongs to some $\si$-image of $M$ and so the orbit of $y$ coincides
with the orbit of $z$. However this is impossible as by the
construction $z\in H$ while $H$ cannot contain points of
$\si$-images of $\ell$. The rest of the lemma follows from the above
and the already obtained description of quadratic invariant gaps of
$\si$.
\end{proof}

By the above proven lemmas, each gap $W=G(c)$ of periodic type has a
periodic major $M(W)=\ol{xy}$ of period $n_c$ with endpoints in
$L(c)$; moreover, $x$ and $y$ are the closest in $L(c)$ points to
the endpoints of $c$ that are $\si^{n_c}$-fixed (in fact, they are
periodic of period $n_c$).

\begin{lem}\label{cantor}
If an invariant quadratic gap $W$ is either of regular critical or of
periodic type, then $W'$ is a Cantor set. If $W=G(c)$ is an extended
caterpillar gap, then $W'$ is the union of a Cantor set and a
countable set of isolated points, all of which are preimages of the
endpoints of $c$.
\end{lem}

\begin{proof}
In the regular critical and periodic cases, it suffices to prove
that the set $W'$ has no isolated points.
Indeed, by Lemma~\ref{l:fx-maj}, an isolated point in $W'$ must
eventually map to an endpoint of $M(W)$. Thus it remains to show
that the endpoints of $M(W)$ are not isolated. This follows because
the endpoints of $M(W)$ are periodic, and suitably chosen pullbacks
of points in $W'$ to $W'$ under the iterates of the remap of $W'$
will converge to the endpoints of $M(W)$. The case of an extended
caterpillar gap follows from Lemma~\ref{picper}.
\end{proof}

\subsection{Canonical laminations of invariant quadratic gaps}
\label{s:canlamqua}

Let us associate a specific lamination with each invariant quadratic gap.
We do this in the spirit of \cite{thu85}, where pullback laminations are defined for maximal collections of critical leaves.
Since $\sim$-classes are finite, an invariant lamination cannot
contain any \ca{} gaps. Hence we consider only regular critical gaps
and gaps of periodic type.

Let $U$ be a stand alone quadratic invariant gap of regular critical
type with critical major $M(U)$. Edges of $U$ have uniquely defined
iterated pullbacks disjoint from $U$ which define an invariant
lamination. More precisely, we define a lamination
$\sim_U$ as follows: $a\sim_U b$ if 
there is $N\ge 0$ such that $\si^N(a)$ and $\si^N(b)$ are endpoints
of the same edge of $U$, and the set $\{\si^i(a),\si^i(b)\}$ is not
separated by $U$ for $i=0$, $\dots$, $N-1$. Loosely, one can say that
points $a, b$ ``travel'' together visiting the same holes of $U$
until at some moment they simultaneously map to the endpoints of
an edge of $U$. Clearly, all
$\sim_U$-classes are either points or leaves and $\sim_U$ is an
invariant lamination. Now, let $U$ be of periodic type and $V$ be
its vassal. Define a lamination $\sim_U$ as follows: $a\sim_U b$ if
there exists $N\ge 0$ such that $\si^N(a)$ and $\si^N(b)$ are
endpoints of the same edge of $U$ or the same edge of $V$, and the
chord $\ol{\si^i(a)\si^i(b)}$ is disjoint from $U\cup V$ for $i=0$,
$\dots$, $N-1$. Note that $V$ is a gap of $\sim_U$. It is easy to
check that the canonical lamination of a quadratic periodic gap $U$
does not have periodic non-degenerate classes that are not edges of
$U$.

\begin{lem}\label{l:canlam1}
If $U$ is a stand alone invariant quadratic gap of regular critical type, then
$\sim_U$ is the unique invariant lamination such that $U$ is
one of its gaps.
\end{lem}

In Definition \ref{d:co-exist-i}, we defined coexistence of a gap
(of some unspecified lamination) and a lamination. The definition
holds verbatim if a gap of some lamination is replaced with a stand
alone gap. We also defined coexistence of two laminations.

\begin{lem}
  \label{l:canlam1s}
  Suppose that a cubic invariant lamination $\sim$ coexists with
  a stand alone invariant quadratic gap $U$ of regular critical type.
  Then $\sim$ also coexists with the canonical lamination $\sim_U$ of $U$.
\end{lem}

\begin{proof}
  Suppose that a leaf $\ell$ of $\sim$ crosses a leaf $\ell_U$ of
  $\sim_U$ in $\disk$. By the assumption of the lemma, both $\ell$
  and $\ell_U$ must have their endpoints in the closure of some hole
  of $U$. Every hole of $U$ maps one-to-one onto its image. It
  follows that $\si(\ell)$ and $\si(\ell_U)$ also intersect in
  $\disk$. However, any leaf of $\sim_U$ eventually maps to an edge
  of $U$, a contradiction.
\end{proof}

\begin{proof}[Proof of Lemma \ref{l:canlam1}]
Suppose that $\sim$ is a cubic invariant lamination and $U$ is a gap
of $\sim$; then $\sim$ coexists with $U$, hence, by Lemma
\ref{l:canlam1s}, the lamination $\sim$ coexists with the canonical lamination
$\sim_U$. If a leaf $\ell$ of $\sim$ is not a leaf of $\sim_U$, then
$\ell$ is in some pullback of $U$ with respect to $\sim_U$. Hence the leaf $\ell$ eventually
maps to $U$. Since $U$ is a gap of $\sim$, the leaf $\ell$
eventually maps to an edge of $U$. By definition it follows that
$\ell$ is a leaf of $\sim_U$. The definition of a lamination now
implies that all leaves of $\sim_U$ are leaves of $\sim$.
\end{proof}

The proof of Lemma~\ref{l:canlam2} is similar to that of Lemmas
~\ref{l:canlam1} and \ref{l:canlam1s}.

\begin{lem} \label{l:canlam2}
Let $U$ be an invariant quadratic gap of periodic type. Then
$\sim_U$ is the unique invariant lamination such that $U$ and the
vassal $V(U)$ are its gaps. If a cubic invariant lamination $\sim$
coexists with $U$ and $V(U)$, then $\sim$ coexists with the canonical
lamination $\sim_U$ of $U$.
\end{lem}

\begin{figure}[htp]
\includegraphics[width=6cm]{qua-per1.eps}
\includegraphics[width=6cm]{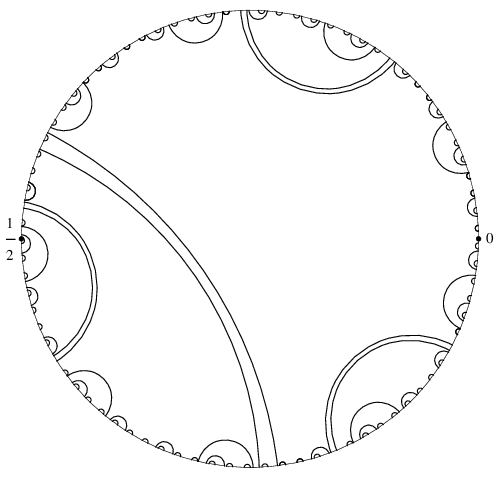}
{\caption{Left: a quadratic invariant gap $U$ of periodic type and
its vassal $V$. Right: its canonical lamination.}}
\end{figure}

\subsection{Tuning}\label{ss:coetun}

In this subsection we discuss the notion of coexistence of two
laminations and make it more explicit. We will define the notion of
\emph{tuning} which is stronger than coexistence of laminations.
Right after Definition~\ref{d:lamset} we introduced a monotone map $\psi_G:\bd(G)\to\uc$; for an invariant stand alone gap $G$ of degree $k>1$, we have that $\si_d|_{\bd(G)}$ is semiconjugate to $\si_k$ by means of $\psi_G$.
If a lamination $\sim$ coexists with
$G$, then we want to obtain an induced lamination $\psi_G(\sim)$
which is invariant under $\si_k$. We need the following notions.

\begin{defn}[\cite{bmov11}]\label{d:sibling}
Let $\lam$ be a geo-la\-mi\-na\-tion such that the $\si_d$-images of
its leaves are again leaves of the same geo-lamination. Any two
disjoint leaves of $\lam$ with the same $\si_d$-image are called
\emph{sibling leaves}, or \emph{siblings}. Suppose that any leaf of
$\lam$ has at least one $\si_d$-preimage, and, for any leaf $\ell_1$
of $\lam$ with non-degenerate $\si_d$-image, there are $d-1$ leaves
$\ell_2$, $\dots$, $\ell_d$ of $\lam$ such that the leaves $\ell_i$,
$i=1$, $\dots$, $d$ are pairwise disjoint and map to
$\si_d(\ell_1)$. This property is called the \emph{Sibling Property}
and $\lam$ is then called a \emph{sibling invariant lamination}.
\end{defn}

Let us explain the difference between Definition~\ref{d:sibling} and
the so-called \emph{leaf invariance} (see
Definition~\ref{geolaminv}(1)). In Definition~\ref{geolaminv}(1), one
begins with a non-degenerate leaf $\ell$ and postulates the
existence of $d$ pairwise disjoint preimage-leaves of $\ell$. In
Definition~\ref{d:sibling} we begin with any leaf $\ell_1$ whose
image is a  non-degenerate leaf and postulate the existence of $d-1$
pairwise disjoint (and disjoint from $\ell_1$) leaves with the same
image $\si_d(\ell_1)$. This is a little stronger than
Definition~\ref{geolaminv}(1) as it does not follow from
Definition~\ref{geolaminv}(1) that we will be able to find siblings
of \emph{any} leaf mapped to $\ell$. A surprising fact  is that this
subtle difference proves to be sufficient to imply all other
properties of invariant geo-laminations.

More precisely, the main result of \cite{bmov11} is that \emph{sibling invariant
laminations are invariant geo-laminations}. An advantage of using
sibling invariant laminations is that checking if a geo-lamination
is sibling invariant requires considering only leaves of the
lamination. Also, geo-laminations generated by laminations are
sibling invariant \cite{bmov11}, and it is proved in \cite{bmov11}
that if sibling invariant laminations $\lam_i$ are such that continua $\lam^+_i$
converge (in the
Hausdorff metric) to a continuum $K$, then in fact there exists a
sibling invariant geo-lamination $\lam$ such that $\lam^+=K$.

Now we consider the union of all leaves of two coexisting
laminations $\sim$ and $\simeq$. By the above, this gives rise to a
sibling invariant lamination $\lam_\sim\cup \lam_\simeq$. To show
that this generates a lamination, we need more tools.

\begin{defn}[Proper lamination \cite{bmov11}]\label{d:proper-lam}
Two leaves with a common endpoint $v$ and the same non-degenerate
image are said to form a \emph{critical wedge $($with vertex $v$}). A
lamination $\lam$ is \emph{proper} if it contains no critical leaf
with a periodic endpoint and no critical wedge with a periodic
vertex.
\end{defn}

A geo-lamination $\lam$ has \emph{period matching property} if any
leaf of $\lam$ with a periodic endpoint of period $n$ is such that
its other endpoint is also of period $n$. Lemma~\ref{l:periodic}
follows from the definitions.

\begin{lem}\label{l:periodic} Suppose that $\sim$ is a lamination
and $\lam_\sim$ is the geo-lamination generated by $\sim$. Then
$\lam_\sim$ has period matching property. Also, any geo-lamination
with period matching property is proper.
\end{lem}

By Lemma~\ref{l:periodic}, given two coexisting laminations $\sim$
and $\simeq$, their geo-laminations $\lam_\sim$ and $\lam_\simeq$
have period matching property. Then $\lam_\sim\cup \lam_\simeq$ also
has period matching property and, hence, is a proper geo-lamination.
Conversely, suppose that $\lam$ is an invariant
geo-lamination. Define an equivalence relation $\approx_\lam$ by
declaring that $x{\approx_\lam}y$ if and only if there exists a
\emph{finite} concatenation of leaves of $\lam$ connecting $x$ and
$y$. Theorem~\ref{t:nowander} specifies properties of
$\approx_\lam$.

\begin{thm}[\cite{bmov11}]\label{t:nowander}
Let $\lam$ be a proper invariant geo-lamination. Then $\approx_\lam$
is an invariant lamination.
\end{thm}

We defined tuning in Definition \ref{d:tuning-i}. Take an invariant
lamination $\simeq$ that coexists with an invariant stand alone
quadratic gap $U$; e.g., it may be that $\simeq$ tunes $U$.
We want to show
that then $\psi_U$ transports $\simeq$ to a quadratic invariant
lamination, which we will denote $\psi_U(\simeq)$. Take a
non-critical leaf $\ell$ of $\simeq$ inside $U$. It has two sibling
leaves which are disjoint. Clearly, one of them, say, $\oa$, is
contained in $U$. If $\psi_U(\oa)$ and $\psi_U(\ell)$ are
non-disjoint, then either $\psi_U(\ell)=\psi_U(\oa)$ is a critical
leaf or the two chords $\psi_U(\ell)$, $\psi_U(\oa)$ form a critical
wedge, a wedge subtended by a diameter. We need to show that the
latter case is impossible. In this case, $\ell$ and $\oa$ share
endpoints with an edge $\ol e$ of $U$. There are no leaves of
$\simeq$ separating $\ell\sm\uc$ from $\oa\sm\uc$ in $\disk$, since
any such leaf would have to cross $\ol e$. Therefore, $\ell$ and
$\oa$ are edges of the same gap $G$ of $\simeq$.
Since some edges of $G$ are sibling leaves, the gap $G$ must be quadratic.
The sibling $\ol e^*$ of $\ol e$ in $G$ is clearly an edge of $U$ connecting the endpoints of $\ell$ and $\ol a$ different from the endpoints of $\ol e$.

But then we have $\psi_U(\ell)=\psi_U(\ol a)$. Thus,
$\psi_U$-images of leaves of $\simeq$ inside $U$ form a sibling
$\si_2$-invariant geo-lamination, cf. \cite[Section 6]{bmov11}.

As $\lam_{\simeq}$ is proper, it follows that $\psi_U(\lam_\simeq)$
is proper. Indeed, by Lemma~\ref{l:periodic}, if an endpoint of a
leaf $\ell$ of $\lam_{\simeq}$ is periodic, then $\ell$ is periodic.
Consider a leaf $\widetilde\ell$ of $\lam_\simeq$ and the leaf $\psi_U(\widetilde\ell)$.
If an endpoint of $\psi_U(\widetilde\ell)$ is periodic, then
$\widetilde\ell$ has a $\si_3$-periodic endpoint, hence $\widetilde\ell$ is periodic, hence $\psi_U(\widetilde\ell)$ is periodic.
Thus the lamination $\psi_U(\lam_\simeq)=\lam$ is proper and, by Theorem~\ref{t:nowander},
one can construct the lamination $\approx_\lam$ which generates a geo-lamination $\lam_{\approx_\lam}$.

Let us compare $\lam$ with $\lam_{\approx_\lam}$ and show that they almost coincide.

\begin{lem}
 Let $\lam$ be a proper quadratic geo-lamination.
Then $\lam\supset \lam_{\approx_\lam}$.
Moreover, $\lam\sm\lam_{\approx_\lam}$ consists of the grand orbit of
a critical leaf and/or the grand orbit of a critical quadrilateral
of $\lam$ that is strictly inside a finite critical gap of $\lam_{\approx_\lam}$.
\end{lem}

\begin{proof}
To prove that $\lam\supset
\lam_{\approx_\lam}$, it suffices to show that the equivalence
relation $\approx_\lam$ produces leaves which the geo-lamination
$\lam$ already contains. By Thurston's No Wandering Triangle theorem
\cite{thu85}, any finite non-precritical gap of a quadratic
lamination is pre-periodic or periodic. Moreover, by \cite{thu85}, the
vertices of a periodic gap must form one cycle. Hence any chord
inside such a gap will cross itself and cannot be a leaf of any
lamination. Therefore, if a gap $G$ of $\approx_\lam$ is not a gap of
$\lam$, then it must be critical or pre-critical. Suppose that $G$ is
critical. If $\lam$ has more than a critical leaf or/and a critical
quadrilateral inside $G$,  then the image of $G$ is still a gap and
has at least one chord inside. However, $\si_2(G)$ cannot be
pre-critical, and, by the previous case, this is impossible.
\end{proof}

Hence, any quadratic proper sibling lamination $\lam$ can be cleaned
(if necessary) by means of removing critical leaves/quadrilaterals
of it contained inside appropriate finite gaps of
$\lam_{\approx_\lam}$ (as described above) as well as removing their
pullbacks. This results into the geo-lamination generated by
$\approx_\lam$. In particular, we can clean the geo-lamination
$\psi_U(\lam_\simeq)$ constructed above and in this way relate it to
a certain quadratic lamination $\asymp$. Strictly speaking, the
lamination $\asymp$ coincides with $\approx_{\psi_U(\lam_\simeq)}$
(first the lamination $\simeq$ generates the geo-lamination
$\lam_\simeq$, then the geo-lamination $\lam_\simeq$ is transported
to the circle by the map $\psi_U$, and then the geo-lamination
$\psi_U(\lam_\simeq)$ generates the lamination
$\approx_{\psi_U(\lam_\simeq)}$), however for brevity in what
follows we will simply denote $\asymp$ by $\psi_U(\simeq)$.

The above arguments allowed us to define the lamination
$\psi_U(\simeq)$. They were based on the fact that $U$ is an
invariant \emph{quadratic} stand alone gap. Literally the same
arguments apply if $U$ is a stand alone periodic quadratic gap
(i.e., a periodic Fatou gap of degree two). Hence, in the periodic
case, given a lamination $\simeq$ that coexists with $U$, we can
also define the lamination $\psi_U(\simeq)$.

Lemma~\ref{l:canlamtun} proves in the case of tuning the claims
established in Lemmas~\ref{l:canlam1}, \ref{l:canlam1s} and \ref{l:canlam2}.
The proof is analogous to the proofs of
Lemma~\ref{l:canlam1} and ~\ref{l:canlam1s} and is left to the
reader.

\begin{lem}\label{l:canlamtun}
Suppose that $\approx$ is a lamination which tunes an invariant
quadratic gap $U$. Then the following holds.

\begin{enumerate}

\item If $U$ is of regular critical type, then $\approx$ tunes the canonical lamination $\sim_U$.

\item If $U$ is of periodic type and $\approx$ tunes $V(U)$ as well,
then in fact $\approx$ tunes the canonical lamination $\sim_U$.

\end{enumerate}

\end{lem}

This provides a more explicit description of how a lamination can
tune
an invariant quadratic gap.


\subsection{Cubic laminations with no rotational gaps or leaves}

We now describe all cubic laminations with no periodic rotational
gaps or leaves (a {\em rotational leaf} is a periodic leaf with
non-refixed endpoints). Recall that for an invariant quadratic gap
of periodic type $U$ we let $M^*(U)$ denote the sibling leaf of $M(U)$ in $V(U)$.
Recall from \cite{bmov11} that,
for any invariant lamination $\sim$, the corresponding
geo-lamination $\lam_\sim$ is sibling invariant.

\begin{lem}\label{l:spec-GvsL}
Let $U$ be a stand alone invariant quadratic gap of periodic type
with major $M=M(U)$ of period $k$. Suppose that $M$ is a leaf of a
lamination $\sim$.
Then the leaf $M^*(U)=M^*$ is a leaf of $\sim$ too.
Moreover, if both $M$ and $M^*$ are edges of a single gap of $\sim$, then this gap coincides with $V(U)$, and if $M$ and $M^*$ are not on the boundary of a single gap of $\sim$, then $\sim$ has a rotational gap or leaf in $V(U)$.
\end{lem}

Observe that we do not assume $U$ to be a gap of $\sim$.

\begin{proof}
Since $M$ is a leaf of $\sim$, then by the Sibling Property so is $M^*$.
Now, if there is a gap $G$ of $\sim$ such that $M$ and $M^*$ are edges of $G$, then, by definition of $V(U)$, we see that $G\subset V(U)$ is a $\si^k_3$-invariant gap.
Applying the map $\psi_{V(U)}$, we get a $\si_2$-invariant gap $\psi_{V(U)}(G)$ that contains the angle $0$.
Clearly, then $\psi_{V(U)}(\bd(G))=\uc$ and hence $G=V(U)$.

Suppose that $M$ and $M^*$ are not on the boundary of a single gap of $\sim$.
Then there must be a leaf $\ell$ of $\lam_\sim$ that separates $M\sm \uc$ from $M^*\sm \uc$ in $\disk$. Note that one of its siblings $\ell^*$ is also contained in the strip between $M$ and $M^*$.
We claim that $(\ell\cup\ell^*)\cap (M\cup M^*)=\0$.
To see this, note that if $\ell$ and $M$ share an endpoint, then $\ell$ is not critical because $\lam_\sim$ is proper.
Hence $\si^k(\ell)\cup\ell\cup M$ is a tripod, which is impossible.
Indeed, $\si^k(\ell)$ cannot coincide with $\ell$, since there are no $k$-periodic points between $M$ and $M^*$.
Hence $M$ is approached by leaves of $\sim$ separating $M$ from $M^*$ (as such we can choose pullbacks of $\ell$ or $\ell^*$).
Choose one such leaf $\bq$ of $\sim$ close enough to $M$ and then choose the leaf $\bq^*$ with the same $\si$-image as $\bq$, separating $\bq$ from $M^*$ ($\bq^*$ is a
leaf of $\sim$ by the Sibling Property).
By Theorem~\ref{t:fxpt}, there exists a $\si^k$-fixed gap or leaf $Q$ in the closed strip
$\ol{S}$ where $S=S(\bq,\bq^*)$ is the open strip between $\bq$ and $\bq^*$.
It follows that the entire orbit of $Q$ is located in the same parts of the circle, where the orbit of $V(U)$ is located, which implies that $Q\subset V(U)$.
Applying $\psi_{V(U)}$ to $Q$ and using well-known facts about quadratic laminations and their invariant sets, we see that $Q$ is rotational, as desired.
\end{proof}

Let us characterize laminations with no rotational sets.

\begin{lem}\label{l:pc0}
Suppose that a cubic invariant lamination $\sim$ has no periodic
rotational gaps or leaves. Then either $\sim$ is empty, or it
coincides with the canonical lamination of an invariant quadratic
gap.
\end{lem}

\begin{proof}
Suppose that a non-empty lamination $\sim$ has no rotational sets.
Then by Theorem \ref{t:fxpt} there is an invariant  gap $U$ of
degree $1<k\le 3$. If $k=3$, then $\sim$ is empty, hence $k=2$. If
$U$ is of regular critical type, then, by Lemma~\ref{l:canlam1},
the lamination $\sim$
is the canonical lamination $\sim_U$. Let $U$ be of periodic type.
Since $M(U)$ is a leaf of $\lam_\sim$, then by
Lemma~\ref{l:spec-GvsL}, so is $M^*(U)$. Now, if there is a gap $G$
of $\sim$ such that $M(U)$ and $M^*(U)$ are edges of $G$, then, by
Lemma~\ref{l:spec-GvsL}, we have $G=V(U)$, and, by Lemma~\ref{l:canlam2},
the lamination $\sim$ coincides with the canonical lamination $\sim_U$. Suppose
that $M(U)$ and $M^*(U)$ are not contained in the same gap of
$\sim$. Then, by Lemma~\ref{l:spec-GvsL}, there exists a rotational
gap or leaf of $\sim$, a contradiction.
\end{proof}

\subsection{Coexistence of quadratic invariant gaps and other laminational sets}\label{ss:afew}

Here we show how invariant quadratic gaps can coexist with each
other as well as how they can coexist with other laminational sets
(we consider gaps of laminations, i.e., invariant quadratic gaps
which are regular critical or of periodic type). We are motivated
here by the desire to provide a model for specific families of
laminations and laminational sets (such as the family of all
quadratic invariant gaps) which should be helpful in the description
of the entire  cubic Mandelbrot set $\M_3$. Some of these lemmas are
used in \cite{bopt13}.

Let us discuss two special quadratic invariant gaps which often play
the role of exceptions to the claims proven below. Let
$\ol{0\frac12}=\di$ be the unique chord in $\disk$ with
$\si$-invariant endpoints. Let $\fg_a$ be the convex hull of all
points with orbits \textbf{a}bove $\di$ and $\fg_b$ be the convex
hull of all points with orbits \textbf{b}elow $\di$. Then $\di$ is
the major of both gaps. However as a major of $\fg_b$ it should be
viewed so that the positive direction on $\di$ is from $0$ to
$\frac12$, and if $\di$ is considered as the major of $\fg_a$, then
the positive direction on $\di$ is from $\frac12$ to $0$. Recall
that when talking about a Jordan curve $K$ which encloses a simply
connected domain $W$  on the plane, by the \emph{positive direction}
on $K$ one means the counterclockwise direction with respect to $W$,
i.e., the direction of a particle moving along $K$ so that $W$ remains on
its left.


Let $U$ be an invariant quadratic gap.
If $U$ is of regular critical type, then we set $M^*(U)=M(U)$;
if $U$ is of periodic type, then we set $M^*(U)$ to be the leaf that is not an edge of $U$ and that has the property $\si(M^*(U))=\si(M(U))$.
We summarize a few simple facts

If a lamination has the gap $U$, it must have the leaf $M^*(U)$.
Let $S_U$ be the closed strip in $\disk$ between $M(U)$ and $M^*(U)$, and set $H(M(U))=H(U)$.
In the regular critical case, we have $\frac13=|H(U)|$; in the periodic case, by
Lemma~\ref{bndcrit}, we have $\frac13<|H(U)|\le \frac12$ with $|H(U)|=\frac12$ only if $M(U)=\di$ (and hence only if $U=\fg_a$ or $U=\fg_b$).
For $M(U)\ne \di$, the arc $H(U)$ has the same endpoints as $M(U)$ and is shorter than $\frac12$; the basis of the gap $U$ is contained in $\uc\sm H(U)$.

Denoting $M(U)$ by $\ol{ab}$, we \emph{always} mean that the direction from $a$ to $b$ along $H(U)$ is positive.
Denote the closed circle arcs from $\bd(S_U)$ by
$L_U,$ $R_U$ with positive direction on $H(U)$ being from $R_U$ to
$L_U$. If $M(U)$ is critical, $S_U=M(U)$ is a chord. Clearly, $M(U)$
determines $R_U$, and  $R_U$ determines $M(U)$:
if $R_U=[a,b^*]$, where $b^*=b-\frac 13$, then $M(U)=\ol{a b}$.
By Lemma~\ref{descru}, the orbit of an endpoint of $M(U)$ avoids $H(U)$; thus, the orbit of $a$ cannot enter $(a, b]\supset (a,a^*]$ (as before, we write $a^*$ for $a+\frac 13$).
Observe that $a$ is periodic and $b^*$ is pre-periodic.
Clearly, $|R_U|=|L_U|=|H(U)|-\frac13$.
Hence $|R_U|\le \frac16$ with $|R_U|=\frac16$ only if $M(U)=\di$.

\begin{lem}\label{l:cofc}
Let $U\ne W$ be invariant quadratic gaps of two different laminations. Then $R_U\cap R_W=L_U\cap L_W=\0$.
\end{lem}

\begin{proof}
If $U$, $W$ are of regular critical type, the claim follows. Let $U$
be of periodic type such that $M(U)$ is of period $m$. Consider a
point $t$ in the interior of $R_U$. Then the analysis of the
dynamics of $\si^m$ on $S_U$ (similar to that of the dynamics of
$\si_2$) implies that $\si^m(t)\in (t, t+\frac13)$
(similar to the statement that $\si_2(s)\in (s,s+\frac12)$ for
every $s\in (0,\frac12)$). Hence $t$ cannot
be the initial endpoint of a major of regular critical type or of
periodic type. Now consider the endpoints of  $R_U$.
Let $R_U=[a,b^*]$ and $H(U)=(a,b)$, where $b^*=b-\frac 13$.
Let us show that $a$ is not the initial point of the major of a quadratic invariant gap $W\ne U$.
Without loss of generality we may assume that $H(W)=(a,\widetilde b^*)$ and $\widetilde b^*\in (a^*,b)$, where $a^*=a+\frac 13$.
Then, as above, $\si^m(\widetilde b^*)\in (a,\widetilde b^*)$,
a contradiction. Also, neither $a$ nor $b^*$ can be the initial point
of a major of regular critical type because both points are initial
points of majors of extended caterpillar type. Since $a$ is
periodic, it is impossible that $[t,a]=R_W$ for some invariant
quadratic gap $W$. This exhausts all possibilities and shows that
$R_U\cap R_W=\0$ if $U\ne W$. Similarly, $L_U\cap L_W=\0$.
\end{proof}

Lemma~\ref{l:cofc} describes a generic type of intersection between two majors. The remaining case is when two majors meet at one
point and are oriented so that their holes are disjoint. In this
case majors must meet at their common endpoint, and pairs of such
majors are rather specific.

\begin{lem}\label{l:majface}
Let $M(U)=\ol{xy}$, $M(W)=\ol{zx}$ be majors of invariant quadratic gaps $U$, $W$.
Assume that $x\in [\frac12, 0]$. Then there are the following cases.
\begin{enumerate}
\item Both $M(U)$ and $M(W)$ coincide with $\di$ oriented in
opposite directions.
\item Both $U$, $W$ are of regular critical type, in which case
$\si(x)=\si(y)=\si(z)$, the forward
orbit of $\si(x)$ is in $[y, z]\cap \bd(\fg_a)$, and
the convex hull of the closure of this orbit is a Siegel gap.
\item Both $U$, $W$ are of periodic type, points $y\in (0, \frac16)$,  $z\in
(\frac13, \frac12)$ belong to the same periodic orbit $P\subset
\bd(\fg_a)$ on which the map acts as a rational rotation.
\end{enumerate}
\end{lem}

\begin{proof}
First let $U$, $W$ be of regular critical type. Then the orbit of
$\si(x)$ is located in the circle arc $[y, z]$ of length $\frac13$.
By \cite{bmmop06}, the closure $T$ of the orbit of $x$ is such that by
collapsing arcs complementary  to $T$ we will semiconjugate $\si|_T$
to an irrational rotation. By the properties of majors,
$0\in (x, y)$ and $\frac 12\in (z, x)$. Hence the orbit of $\si(x)$
is contained in $(0, \frac 12)$. This completes case (2).

Now let $U$ and $W$ be of periodic type.
Observe that if $x=0$ or $x=\frac 12$, then case (1)
takes place. So we may assume that $x\in (\frac12, 0)$. Then $0\in
(x, y)$ and $\frac12\in (z, x)$. Suppose that $\si^k(y)\in [\frac12,
x]$ for some $k$. Then, since the orbit of $x$ never enters $[z,
x)\cup (x, y]$ and $\si^k(y)\ne y$, it follows that $\si^k(x)\in (y,
z)$. Hence $\si^k(M(U))$ separates $\frac12$ from $y$. Since the
orbit of $M(U)$ is on the boundary of $U$ and $\frac12\in \bd(U)$ by
Lemma~\ref{l:fx-maj} and Lemma~\ref{bndcrit}, this leads to a
contradiction and shows that the orbit of $y$ is contained in the
circle arc $[0, \frac12]$, and hence $y\in \bd(\fg_a)$. Moreover,
recall that the orbit of $y$ is contained in $[y, x]$; if $y\in
[\frac13, \frac12]$ then $\si(y)\in (x,y)$, a
contradiction. Hence $y\in [0, \frac16)$.

If $y=0$, then $x=\frac12$ and $z=0$ so that case (1) holds.
If $y\in (0, \frac16)$ then $y^*=y+\frac13\in (\frac13, \frac12)$, and $z\in (y, y^*)$.
Set $x^{**}=x+\frac23$; then $\ol{x^{**}y^*}$ is an edge of $U$. Clearly, $z\in (y, x^{**})$ because, by Lemma~\ref{bndcrit}, the length of the arc $(z, x)$ is between $\frac 13$ and $\frac12$.
Suppose that $\si^k(y)\in (y^*,\frac12)$ for some $k$.
Clearly, $\si^k(x)\notin (x, x^{**})$.
Hence $\si^k(x)\in (y^*, \si^k(y))\subset (z, x)$, a contradiction.
Thus, the orbit of $y$ is contained in $[y, y^*]$.
By \cite{bmmop06}, then the order among points of the orbit of $y$ is the same as the order of points in a periodic orbit under a rational circle rotation.
Since the orbit of $M(U)$ consists of edges of $U$, it follows that the same order is maintained among the points from the orbit of $x$.
Literally the same can be said about the orbit of $z$ and the orbit of $x$.
This implies that the rotation numbers associated with the orbits of $y$ and $z$ are the same.
Since both orbits are contained in $\fg_a$, and $\si|_{\fg_a}$ is semiconjugate to $\si_2$, well-known properties of $\si_2$ imply that $y$ and $z$ belong to the same periodic orbit. So, under the current assumption, case (3) holds.
\end{proof}

\begin{lem}\label{l:petype2}
Let $A$ be a non-degenerate class or an infinite gap of a cubic lamination $\sim$ that has a quadratic invariant gap $U$.
Suppose that $A$ never maps to $U$, $M(U)$ or $V(U)$.
Then $U$ is of periodic type, no image of $A$ intersects $M^*(U)$,
there exists a number $q\ge 0$ and an edge $e$ of $\si_3^q(A)$ such that either $e=M(U)$, or $e$ separates $U$ from $M^*(U)$.
Moreover, there is a leaf $\hat e$ $($possibly equal to $e)$ such that $\si_3(\hat e)=\si_3(e)$, and $\hat e$ separates $U$ from the second critical set $D\ne U$ of $\sim$.
If $A$ is periodic of period $m$, and $k$ is the period of $M(U)$, then $m>k$.
If we have $\si_3^n(A)=U$, and $A$ is located between $M(U)$ and $M^*(U)$, then $n>k$.
\end{lem}

\begin{proof}
Suppose that $U$ is of regular critical type. Then $A$ is a preimage
of $U$ or a preimage of $M(U)$, an edge of $A$ maps to $M(U)$, and
we can set $e=\hat e=M(U)$. The last claim of the lemma follows
because if $U$ is of regular critical type, $A$ can only be periodic
or critical if $A=U$ or $A=M(U)$.

Let $U$ be of periodic type with vassal $V.$ If $\si^l(A)=U$, or
$\si^l(A)=M(U)$, or $\si^l(A)=V$ (by Lemma~\ref{l:canlam2}, this is
true for a suitable choice of $l$ if $\sim$ is the canonical
lamination of $U$), then set $q=l, e=\hat e=M(U)$. Assume now that
$\sim$ is not the canonical lamination of $U$, and $A$ never maps to
$U$, $V$ or $M(U)$. Let us show that no image of $A$ intersects $M^*(U)$.
Indeed, if an image $T$ of $A$ contains an endpoint of $M^*(U)$ then the properties
of laminations imply that either $T=M^*$, or $T=V$, or $T$ is a gap with an edge $M^*(U)$ which maps
onto $U$ after $k$ more steps. Since this contradicts our assumptions, it remains
to consider the case when an image $T$ of $A$ will cross $M^*(U)$. However
then it is easy to see that $\si^K(T)$ will cross $M$, a contradiction.
So, no image of $A$ intersects $M^*(U)$.

Then an edge $\ell$ of $A$ connects two points $t_1$, $t_2$ of the
boundary of a gap $T$ of the canonical lamination $\sim_U$ of $U$
and passes through the interior of $T$. This includes the
possibility that $\ell$ crosses two edges of $T$ in $\disk$. We show
that $\si^q(A)$ separates $M(U)$ from $M^*(U)$ for some $q$. Indeed,
as long as $T$ maps onto its image one-to-one, the images of $\ell$
connect the corresponding images of $t_1, t_2$. By the properties of
$\sim_U$ there is the least $n$ with $\si^n(T)=U$ or $\si^n(T)=V$.
The leaf $\si^n(\ell)$ connects two points of $\bd(\si^n(T)))$ and
passes through the interior of $\si^n(T)$. Hence $\si^n(T)\ne U$,
the leaf $\si^n(\ell)$ connects two points of $\bd(V)$ and passes
through the interior of $V$. Let $\psi_V$ collapse all edges of $V$;
then $\psi_V$ semiconjugates $\si^k|_{\bd(V)}$ to $\si_2$ and maps
$\si^n(\ell)$ to a chord inside $\disk$. By properties of $\si_2$,
the chord $\si^q(\ell)=e$ separates $M(U)$ and $M^*(U)$ for some
$q\ge n$. The claim about $\hat e$ immediately follows (the critical
set $D\ne U$ of $\sim$ is located between $e$ and $\hat e$).

Let us prove the last claims of the lemma. Let $A$ be periodic of
period $m$ and $A\notin \{U, M(U), V(U)\}$. We claim that then for
all $i$ the set $\si_3^i(A)$ cannot have $M(U)$ as an edge. Indeed,
suppose otherwise. Then $\si_3^i(A)$ must be periodic of period $k$.
It follow from the definition of $V(U)$ that $A\subset V(U)$.
Applying the map $\psi_V$ defined above we see that $\psi_V(A)$ must
be the unique fixed point of $\si_2$. Hence $\si_3^i(A)=M(U)$, a
contradiction.

By the above there exists an image $B$ of $A$ and an edge $e$ of $B$
such that $e$ separates $M(U)$ from $M^*(U)$. Hence for $k$ steps
images of $B$ will be located ``behind'' the corresponding images of
$M(U)$ and $\si_3^k(B)\ne B$ as otherwise $\si_V(B)$ will have to be
the unique fixed point of $\si_2$, a contradiction. Therefore $k<m$
as desired.  A similar argument proves the very last claim of the
lemma.
\end{proof}

Now let us study which laminations with invariant quadratic gaps can share 
a non-degenerate class or an infinite gap.

\begin{lem}\label{l:petype1}
Let  $\sim_i$ $(i=1,2)$ be two laminations.
Suppose that $A$ is a non-degenerate class or an infinite gap of both laminations, and that each $\sim_i$ has a quadratic invariant gap $U_i$, $i=1, 2$.
Then $U_1=U_2$ except for the case when $A$ is a gap or a class of the lamination that has both $\fg_a$ and $\fg_b$ as gaps;
in the latter case, $U_1$ and $U_2$ may coincide with $\fg_a$ and $\fg_b$ $($in any order$)$.
\end{lem}

\begin{proof}
Suppose that $A$ is a non-degenerate class or an infinite gap of a lamination $\sim$ that has an invariant quadratic gap $U$.
We show how to recover $U$ knowing $A$, except in the case, where both $\fg_a$ and $\fg_b$ are gaps of $\sim$.
This will imply the statement of the lemma.

For every leaf or gap $T$ of $\sim$, let $||T||$ denote the length of the largest hole of $T$.
Set $\mu=\inf_{n\ge 0} ||\si^n(A)||$.
There may or may not be a leaf or gap $B$ in the forward orbit of $A$ with the property $||B||=\mu$.
If there is no such $B$, then there exists a leaf $\ell$ of $\sim$ such that $||\ell||=\mu$, and
there is a sequence of leaves or gaps in the forward orbit of $A$ accumulating on $\ell$.
In this latter case, we set $B=\ell$.
Suppose that no image of $A$ has $\ol{0\frac12}$ as its edge, and let $H$ be the largest hole of $B$.
We claim that, in this case, $U'$ can be recovered as the set of all points in $H$, whose forward orbits stay in $H$.

To prove the claim, let $M$ denote the major of $U$ and, as before, let $M^*$ denote its sibling not in $U$.
Observe that, by our assumptions and by Lemma \ref{l:spec-GvsL}, both $M$ and $M^*$ are leaves of $\sim$.
Let us show that either $B=M$ or $B$ is in the strip between $M$ and $M^*$.
Indeed, by Lemma \ref{l:petype2}, images of $A$ are located in non-major holes of $U$, or in holes of $U\cup M^*(U)$, or separate $M(U)$ from $M^*(U)$.
If an image $T$ of $A$ is located in a non-major hole of $U$, then, by Lemma~\ref{l:bndcrit1}, we have $||T||>\frac56$.
If an image $T$ of $A$ is located in one of the three remaining holes of $U\cup \M^*(U)$ but not in the hole of $U\cup M^*(U)$ separated from $U$ by $M^*(U)$, then, again by Lemma~\ref{l:bndcrit1}, we have $||T||>\frac56$.
Finally, since, by Lemma \ref{l:petype2}, there are images of $A$ separating $M(U)$ from $M^*(U)$, it follows that $B$ cannot be separated from $U$ by $M^*(U)$.
We conclude that either $B=M$ or $B$ is in the strip between $M$ and $M^*$.
In either case it follows that $U'$ coincides with the set of points in the largest hole
$H$ of $B$, whose forward orbits stay in $H$.
Observe that even if $||B||=\frac12$, the set $U'$ is well-defined.
The remaining case when some image of $A$ equals $\ol{0\frac12}$ is immediate.
\end{proof}


\section{Invariant rotational sets}\label{s:rotgaps}

Fix an invariant rotational laminational set $G$. There are one or
two majors of $G$. We classify invariant rotational gaps by types.
This classification mimics Milnor's classification of
hyperbolic components in slices of cubic polynomials and quadratic
rational functions \cite{M, miln93}. Namely, we say that
\begin{itemize}
  \item the gap $G$ is of type A (from ``Adjacent'') if $G$ has only one
  major (whose length is at least $\frac 23$);
  \item the gap $G$ is of type B (from ``Bi-transitive'') if $G$
  has two majors that belong to the same periodic cycle;
  \item the gap $G$ is of type C (from ``Capture'') if it is not type B,
  and one major of $G$  eventually maps to  the other major of $G$;
  \item the gap $G$ is of type D (from ``Disjoint'') if there are two
  majors of $G$, whose orbits are disjoint.
\end{itemize}
Clearly,  it follows from the definitions that finite rotational
gaps cannot be of  type C (since then $\si|_{\bd(G)}$ is not
one-to-one). Also, if $G$ is of type B, then $\si|_{\bd(G)}$ is
one-to-one, and hence $G$ must be finite.

\subsection{Finite rotational sets}\label{s:finrotset}
A classification of finite rotational sets (under the name of fixed
point portraits) can be found in \cite{gm93}. We now give some examples
illustrating a part of this classification concerning the degree 3
case.

Let $G$ be a finite invariant rotational set (as we fix $G$ in this
section, we may omit using $G$ in the notation). By \cite{kiwi02},
there are at most two periodic orbits (of the same period denoted in
this section by $k$) forming the set of vertices of $G$. If vertices
of $G$ form two periodic orbits, points of these orbits alternate on
$\uc$.

\begin{example}\label{fingap1}
Consider the invariant rotational gap $G$ with vertices
$\frac7{26}$, $\frac4{13}$, $\frac{11}{26}$, $\frac{10}{13}$,
$\frac{21}{26}$ and $\frac{12}{13}$. This is a gap of type D.
The major leaf $M_1$ connects $\frac{12}{13}$ with $\frac{7}{26}$ and
the major leaf $M_2$ connects $\frac{11}{26}$ with $\frac{10}{13}$.
These major leaves belong to two distinct periodic orbits of edges of $G$. The major hole $H_G(M_1)$ contains $0$ and the major hole
$H_G(M_2)$ contains $\frac12$.
\end{example}

\begin{figure}[htp]
\includegraphics[width=6cm]{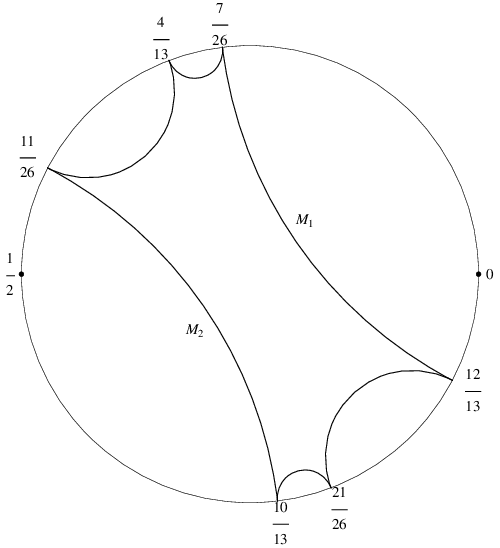}
\includegraphics[width=6cm]{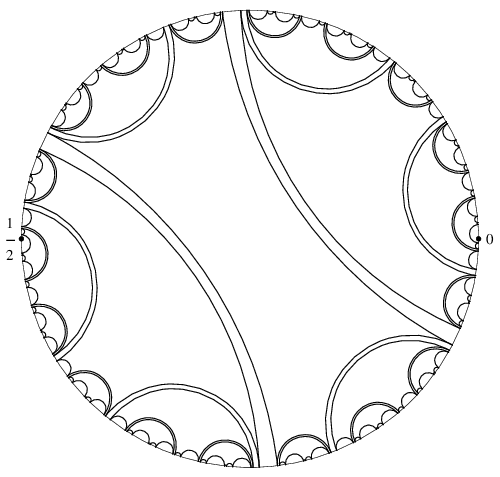}
\caption{The rotational gap described in Example \ref{fingap1} and
its canonical lamination.}
\end{figure}

The next example can be obtained by considering \emph{one} periodic orbit
of vertices in the boundary of the gap from Example~\ref{fingap1}

\begin{example}\label{fingap2}
Consider the finite gap $G$ with vertices $\frac7{26}$,
$\frac{11}{26}$ and $\frac{21}{26}$. This is a gap of type B. The
first major leaf $M_1$ connects $\frac{21}{26}$ with $\frac{7}{26}$
and the second major leaf $M_2$ connects $\frac{11}{26}$ with
$\frac{21}{26}$. The edges of $G$ form one periodic orbit to which
both $M_1$ and $M_2$ belong. The major hole $H_G(M_1)$ contains $0$
and the major hole $H_G(M_2)$ contains $\frac12$.
\end{example}

\begin{figure}[htp]
\includegraphics[width=6cm]{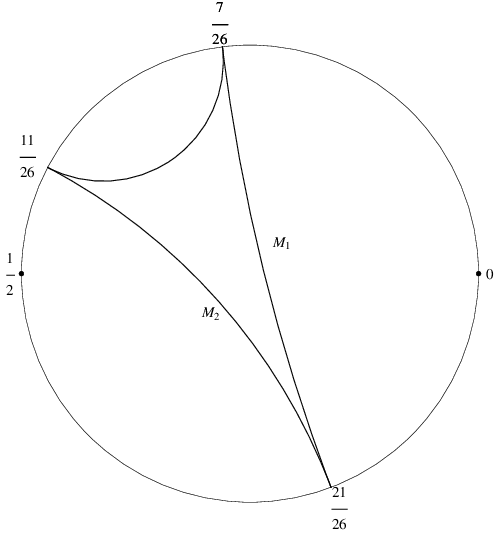}
\includegraphics[width=6cm]{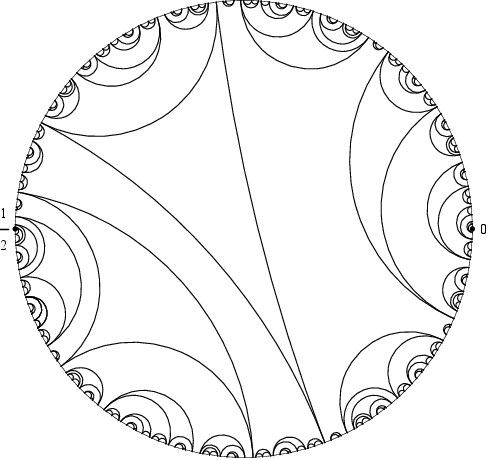}
\caption{The rotational gap described in Example \ref{fingap2} and
its canonical lamination.}
\end{figure}

\begin{example}\label{fingap3}
Consider the finite gap $G$ with vertices $\frac1{26}$,
$\frac{3}{26}$ and $\frac{9}{26}$.
This is a gap of type A.
The only major leaf $M=M_1=M_2$ connects $\frac{9}{26}$ with $\frac{1}{26}$.
The edges of $G$ form one periodic orbit to which $M$ belongs.
The major hole $H_G(M)$ contains $0$ and $\frac12$ and is longer than $\frac23$.
\end{example}

\begin{figure}[htp]
\includegraphics[width=6cm]{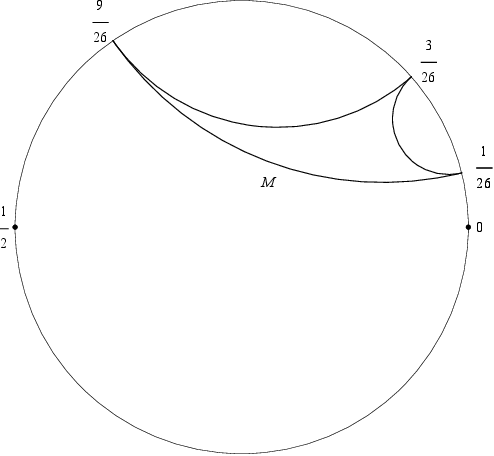}
\includegraphics[width=6cm]{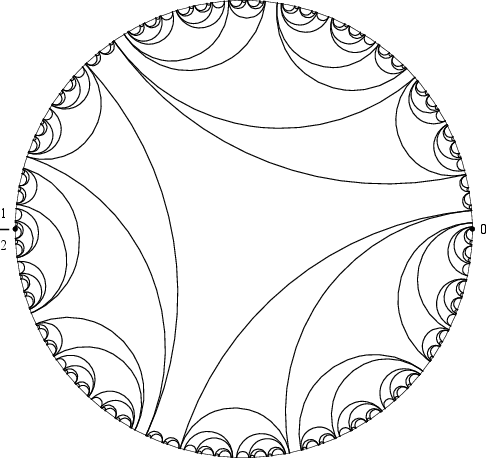}
\caption{The rotational gap described in Example \ref{fingap3} and
its canonical lamination.}
\end{figure}

Let $G=\ell=\ol{ab}$ be an invariant leaf. We can think of $G$ as a
gap with empty interior and two edges $\ol{ab}$ and $\ol{ba}$, and
deal with all finite 
invariant 
sets in a unified way.
Let us list all non-degenerate invariant leaves $\ol{ab}$.
Either points $a$, $b$ are fixed, or they form a two-periodic orbit.
In the first case, we have the leaf
$\ol{0\frac12}=\di$, in the second case, we have one of the leaves
$\ol{\frac18\frac38}$, $\ol{\frac14\frac34}$, $\ol{\frac58\frac78}$.
Informally, we regard $\di$ as an invariant rotational set of type D
(even though its rotation number is $0$).

Let $G$ be a finite invariant laminational set with $m$ edges
$\ell_0$, $\dots$, $\ell_{m-1}$. For each $i$, let $\fg_i$ be the
convex hull of all points $x\in \ol{H_G(\ell_i)}$ with $\si^j(x)\in
\ol{H_G(\si^j(\ell_i))}$ for every $j\ge 0$ (compare this to the
definition of a vassal in Section \ref{s:qgaps}). It is
straightforward to see that $\fg_i$ are infinite stand alone gaps
such that $\fg_i$ maps to $\fg_j$ if $\ell_j=\si(\ell_i)$. These
gaps are called the {\em canonical Fatou gaps attached to $G$}. The
gap $\fg_i$ is critical if and only if the corresponding edge
$\ell_i$ is a major.

\subsection{Canonical laminations of finite invariant rotational sets}
\label{s:canlamfin}

To every finite invariant rotational set $G$, we associate its
\emph{canonical} lamination $\sim_G$.

Suppose first that $G$ is of type B or D. Then by definition there
are two major edges of $G$, which are denoted by $M_1$ and $M_2$.
Let $H_1$ and $H_2$ be the corresponding holes. By Lemma
\ref{l:fx-maj}, we may assume that $0\in H_1$ and $\frac 12\in H_2$.
Since $M_1$ and $M_2$ are periodic, their lengths are strictly
greater than $\frac 13$. Let $U_1$ and $U_2$ be the canonical Fatou
gaps attached to $M_1$ and $M_2$, respectively. Thurston's pullback
construction \cite{thu85} yields an invariant lamination formed by
the pullbacks of $G$ disjoint from the interiors of $U_1$ and $U_2$.
More precisely, we can define a lamination $\sim_G$ as follows: two
points $a$ and $b$ on the unit circle are equivalent if there exists
$N\ge 0$ such that $\si^N(a)$ and $\si^N(b)$ are vertices of $G$,
and the chords $\ol{\si^i(a)\si^i(b)}$ are disjoint from $G$ and
from the interior of $U_1\cup U_2$ for $i=0$, $\dots$, $N-1$. It is
straightforward to check that $\sim_G$ is indeed an invariant
lamination. This lamination is called the \emph{canonical lamination
associated with $G$}.

Assume now that $G$ is of type A. Let $M$ be the major edge of $G$,
and $U$ the corresponding canonical Fatou gap attached to $G$ at
$M$. The canonical lamination $\sim_G$ of $G$ is defined similarly
to those for types B and D. Namely, two points $a$ and $b$ on the unit circle
are equivalent with respect to $\sim_G$ if there exists $N\ge 0$
such that $\si^N(a)$ and $\si^N(b)$ are vertices of $G$, and the
chord $\ol{\si^i(a)\si^i(b)}$ is disjoint from $G$ and from the
interior of $U$ for $i=0$, $\dots$, $N-1$.

Lemma~\ref{l:gapsuniq} is similar to Lemma~\ref{l:canlam1}. It is
based on Thurston's pullback construction of laminations.

\begin{lem}\label{l:gapsuniq}
Suppose that $\sim$ has a finite invariant gap $G$ and all the
canonical Fatou gaps attached to $G$ are gaps of $\lam_\sim$. Then
$\sim$ coincides with the canonical lamination of $G$.
\end{lem}

\begin{lem}
\label{l:spec-LvsG1}
Suppose that a cubic invariant lamination $\sim$ has a finite invariant gap $G$ of type D. If a canonical Fatou gap $U$ of $G$ is not a gap of $\sim$, then $\sim$ has a rotational gap or leaf in $U$.
Thus, if $\sim$ is not the canonical lamination of $G$, then $\sim$ has a rotational periodic gap or leaf in a canonical Fatou gap attached to $G$.
\end{lem}

\begin{proof}
Clearly, a major of a gap $G$ of type D satisfies the conditions of
Lemma~\ref{l:posholes} and can be viewed as the major of some
invariant quadratic gap $W$ of periodic type. Then the canonical
Fatou gap $U$ attached to $G$ coincides with the vassal gap $V(W)$
of $W$. Hence the rest of the lemma follows from
Lemmas~\ref{l:spec-GvsL} and \ref{l:gapsuniq}.
\end{proof}

\subsection{Irrational invariant gaps and their canonical laminations}\label{s:irrgaps}

The description of irrational gaps is close to that of finite
laminational sets. In Subsection~\ref{s:irrgaps} we fix an
irrational rotation number $\tau$.

Let $G$ be an invariant Siegel gap of rotation number $\tau$.
Then $G$ may have one or two critical
majors of length $\frac13$, or one critical major of length
$\frac23$. It is also possible that $G$ has a non-critical major.
However, a non-critical major eventually maps to the critical major
by Lemma \ref{l:maj} (in this case, $G$ is of type C). Thus an
infinite rotational gap $G$ can have type A, C or D.

We now construct the canonical lamination for a Siegel gap $G$ of
type D with critical edges $L$ and $M$. Consider well-defined
pullbacks of $G$ attached to $G$ at $L$ and $M$. Then apply
Thurston's pullback procedure to these gaps. As holes in the union
of bases of these gaps are shorter than $\frac13$, the pullbacks of
the gaps converge in diameter to $0$. Alternatively, we can define
$\sim_G$ as follows: two points $a$ and $b$ in the unit circle are
equivalent if there exists $N>0$ such that $\si^N(a)$ and $\si^N(b)$
lie on the same edge of $G$, and the chords $\ol{\si^i(a)\si^i(b)}$
are disjoint from $G$ for $i=0$, $\dots$, $N-1$.

We will not define canonical laminations of type A or C Siegel gaps.

\begin{lem}
 \label{l:type1Sie}
 Suppose that $G$ is a type D stand alone invariant gap of Siegel type, and $\sim$ is an invariant cubic lamination with gap $G$.
 Then $\sim$ coincides with $\sim_G$.
\end{lem}

The proof goes almost verbatim as in Lemma \ref{l:canlam1}. It is
based on Thurston's pullback construction of laminations.

\section{The description of the Combinatorial Main
Cubioid}\label{s:descricu}

If $G$ is a finite $\si_2$-invariant rotational set, we define the
{\em canonical lamination} of $G$ as the only quadratic invariant
lamination with a cycle of Fatou gaps attached to edges of $G$ (it
represents a parabolic quadratic polynomial $z^2+c$, whose parameter
$c$ belongs to the Main Cardioid). Similarly, if $G$ is a stand
alone invariant Siegel gap with respect to $\si_2$, we define the
\emph{canonical lamination} of $G$ as the unique quadratic invariant
lamination that has $G$ as its gap.

\begin{prop}\label{p:coremin-qua}
A non-empty quadratic lamination $\sim$ with at most one periodic
$($hence fixed$)$ rotational set $G$ coincides with the canonical lamination of $G$. The Combinatorial Main Cardioid $\car^c$ consists of quadratic laminations with at most one periodic rotational set.
\end{prop}

\begin{proof}
By Theorem~\ref{t:fxpt}, the lamination $\sim$ has an invariant rotational set $G_2$.
If $G_2$ is a Siegel gap, then $\sim$ is the canonical lamination
of $G_2$ because all pullbacks of $G_2$ are uniquely defined (cf.
Lemma \ref{l:canlam1}). Let $G_2$ be a finite gap or leaf.
By \cite{kiwi02}, the laminational set $G_2$ has a unique major $M_2$ separating $G_2$
from $0$, and the edges of $G_2$ form one cycle (of period $r$). Let
$V_2$ be the Fatou gap of the canonical lamination of $G_2$ which
has $M_2$ as one of its edges.
Let $M^*_2$ be the sibling of $M_2$;
then $M^*_2$ is an edge of $V_2$. By Theorem~\ref{t:fxpt} and
because $\sim$ has at most one periodic rotational gap or leaf, the
strip between $M_2$ and $M^*_2$ contains a $\si^r_2$-invariant Fatou
gap $U$ of $\sim$ of degree greater than 1. In fact, $r$ is the
period of $U$ as $U$ passes through every hole of $G_2$ before returning.
Hence $U\subset V_2$, which immediately implies that $U=V_2$, and $\sim$ is the canonical lamination of $G_2$.
\end{proof}

Let us go back to cubic laminations $\sim$ (recall that $\si_3$ is denoted by $\si$).
Gaps $U(c)$ for critical chords $c$ are defined right after
Lemma~\ref{l:perpic}.

\begin{lem}\label{l:lamcoex}
Suppose that $\sim$ is a lamination that coexists with two disjoint
critical chords, $c$ and $d$ such that $c$ has non-periodic
endpoints, no leaf of $\sim$ contains an endpoint of $d$, and $d$
intersects no edge of $U(c)$ in $\disk$. Then $\sim$ coexists with
the gap $U(c)$. Moreover, either $\sim$ tunes $U(c)$, or no edge of
$U(c)$ belongs to $\sim$.
\end{lem}

\begin{proof}
We need to show that a leaf of $\sim$ does not intersect an edge of
$U(c)=U$ in $\disk$ unless it coincides with it. By way of
contradiction, suppose that a leaf $\ell$ of $\sim$
intersects an edge $\bj\ne \ell$ of $U$ in $\disk$. Let us show that
then $\si(\ell)$ intersects $\si(\bj)$ in $\disk$. Indeed, there
exists a component of $\ol{\disk}\sm (c\cup d)$ whose closure
contains both $\ell$ and $\bj$. Hence, $\si(\bj)\ne \si(\ell)$.
The only case when $\si(\ell)$ and $\si(\bj)$ do not intersect in
$\disk$ under the circumstances is as follows: $\ell$ and $\bj$
contain distinct endpoints of one of the critical chords, $c$ or
$d$. Let us show that this is impossible.

By the assumptions on $d$ the leaf $\ell$ does not contain an
endpoint of $d$. Suppose that each of the chords $\ell$ and $\bj$
contains an endpoint of $c$. Then an endpoint of $c$ belongs to
$U(c)$, and since $c$ has no periodic endpoints, $c$ is regular
critical. Hence the only edge of $U(c)$ containing an endpoint of
$c$ is $c$, and $\bj=c$. But then $c$ and $\ell$ do not intersect in
$\disk$ because $c$ and leaves of $\sim$ distinct from $c$ are
disjoint inside $\disk$ by the assumption. Thus $\si(\ell)$ and
$\si(\bj)\ne \si(\ell)$ intersect in $\disk$. By induction this
implies that for any $n\ge 0$, $\si^n(\bj)\ne \si^n(\ell)$ intersect
in $\disk$.

Since no image of $\ell$ can intersect $c$ while not coinciding with
$c$, this implies that $c$ cannot be regular critical. Then the
whole orbit of $\bj$ stays strictly on one side of $c$ which implies
that so does the whole orbit of $\ell$. Hence (see
Subsection~\ref{s:invquagap}) the endpoints of $\ell$ belong to
$U'(c)$, and $\ell\ne \bj$ cannot intersect $\bj$ in $\disk$, a
contradiction. Denote the major of $U(c)$ by $M$. The last claim of
the lemma follows from the fact that if $M$ is a leaf of $\sim$ then
$\sim$ tunes $U(c)$ (because $\sim$ is backward invariant and
coexists with $d$) while if $M$ is not a leaf of $\sim$ then no edge
of $U(c)$ can be a leaf of $\sim$ (because $\sim$ is forward
invariant and by Lemma~\ref{descru} which states that all edges of
$U(c)$ are preimages of $M$).
\end{proof}

We also need the following lemma. Observe that a cubic lamination
that has a critical set of degree two must have a second critical
set, also of degree two.

\begin{lem}\label{l:2fatou}
Let $\sim$ be a lamination from $\cu^c$ with a finite invariant gap
$G$ of type A or B such that a cycle $\F$ of Fatou gaps attached to
$G$ at each edge consists of Fatou gaps of degree two. Suppose that
the second critical set $W$ of $\sim$ is infinite.
Then the following holds.
\begin{enumerate}
\item The set $W$ is a periodic Fatou gap of degree 2, and the refixed edge
$\ell$ of $W$ separates the rest of $W$ from $G$.

\item The leaf $\ell$ is the major of a unique quadratic invariant
gap $U$; the lamination $\sim$ tunes the canonical lamination
$\sim_U$ according to a quadratic lamination $\asymp$ from $\car^c$
$($possibly empty$)$; and $W=V(U)$.

\item The gap $U$ is a unique quadratic invariant gap which coexists
with $\sim$ except for the case when $G=\di$ in which case either
$V=\fg_a$ and $U=\fg_b$, or $V=\fg_b$ and $U=\fg_a$.

\end{enumerate}

\end{lem}

\begin{proof}
Clearly, $W$ is a periodic Fatou gap of degree 2. We claim that the
orbit of $W$ is contained in $E\cup W$, where $E$ is the component of
$\ol{\disk}\sm W$ containing $G$. Indeed, otherwise there is $i$
with $\si^i(W)$ contained in the closure of a component $F\ne E$ of
$\disk\sm W$. If $\si^i(W)$ touches $W$ at a vertex, an edge of $W$
and an edge of $\si^i(W)$ must have a common vertex. This implies
that these two edges in fact are edges of a finite rotational set of
$\sim$ distinct from $G$, a contradiction with $\sim$ being from
$\cu^c$. If $\si^i(W)$ and $W$ share an edge, then this edge is
rotational, again a contradiction. Finally, if $\si^i(W)$ is
disjoint from $W$, then, by Theorem~\ref{t:fxpt}, the component $F$ contains a
$\si^i$-invariant rotational gap or leaf of $\sim$ or a
$\si^i$-invariant Fatou domain of $\sim$. The former is impossible
because $\sim\in \cu^c$, and the latter is impossible because the
only two cycles of Fatou domains of $\sim$ are $\F$ and the cycle of
$W$. This proves the claim that the orbit of $W$ is contained in $E\cup W$.

The gap $W$ has a unique edge $\ell$ that separates the rest of $W$
from $G$. Denote the sibling of $\ell$ in $W$ by $\hell$. We
claim that $\ell$ is the refixed edge of $W$ (being quadratic, $W$
has a unique refixed edge). Assume that $\ell$ is not refixed in
$W$. Choose a critical chord $c$ of $W$ with strictly preperiodic endpoints.
Then $c$ divides $W$ in two halves, $A$ and $B$, while $\psi_W(c)$
divides $\uc$ in two halves, $\psi_W(A)$ and $\psi_W(B)$. By the
properties of $\si_2$, each of the two sets $\psi_W(A)$, $\psi_W(B)$ contains an
invariant rotational gap, leaf, or point $\{0\}$, denoted by $T(A)$, $T(B)$,
and one of the sets $T(A)$, $T(B)$ is $\{0\}$.

Since $c$ has strictly preperiodic endpoints, the forward orbit of $c$
intersects both $A$ and $B$.
In particular, $c$ cannot be a regular critical major, and the gap $U(c)$ is
of periodic type. Let $M$ be the major of $U(c)$. Then $M$ separates
$c$ from $G$. Clearly, if $M$ is an edge of $W$, then $M=\ell$, and
if $M$ is not an edge of $W$, then $M$ separates $c$ from $\ell$.
In both cases, we have $M\subset W$.

Assume that $\ell\subset A$. Then by definition
$\psi^{-1}_W(T(A))\subset U(c)$.  Consider the images $M_0$,
$\dots$,  $M_k$ of $M$ in $W$, where $M_0=M$. Their endpoints are
all located in $A$. Hence, by the properties of $\si_2$, their
$\psi_W$-images form either the cycle of vertices of $T(A)$ or the
cycle of edges of $T(A)$. Thus, $M_0, \dots, M_k$ either form a
cycle of edges of $W$, or a cycle of chords of $W$ projected to
$T(A)$ under $\psi_W$. Since in the former case $M_0$ cannot
separate $c$ from $\ell$, either $M=M_0=\ell$ or $M_0$, $\dots$,
$M_{k-1}$ project onto edges of $T(A)$.
In the former case, $\ell$ is refixed as otherwise $M_0$ separates $M_1$ from $G$, and cannot be the major of a quadratic invariant gap of periodic type.
By way of contradiction, assume that $M_0$, $\dots$, $M_{k-1}$ project onto edges of $T(A)$.

The above arguments apply to all choices of $c$ with strictly preperiodic endpoints.
If $T(A)=\{0\}$, then $M$ is a refixed edge of $W$, which as before implies that $M=\ell$. If $T(A)\ne\{0\}$,
then there are at least two images of $M$ in $W$ that separate $G$
from $c$ (two coinciding leaves with different orientations are
considered as different). We see that one of the chords $M_0, \dots,
M_k$ separates $G$ from another one, contradicting properties of
majors of periodic type. This contradiction finally proves that
$\ell=M$ is the major of an invariant quadratic gap of periodic
type.

Now, we have already shown the existence of a quadratic invariant
gap $U=U(c)$ coexisting with $\sim$; moreover, we proved that the
periodic leaf $\ell$ is the major of $U$ and that $U$ is
tuned by $\sim$ (by Lemma~\ref{l:lamcoex}, since $\ell$ is a leaf of
$\sim$, all edges of $U$ are leaves of $\sim$). Clearly,
then $W$ is the vassal of $U$.
It remains to prove that $U$ is the unique quadratic invariant gap coexisting with $\sim$.
Suppose that $Q\ne U$ is a quadratic invariant gap with major
$m$ coexisting with $\sim$.

We will write $V$ for the critical Fatou gap in $\F$.
If $m$ is critical, then it is contained in $V$ or $W$ (since $Q$ and
$\sim$ coexist, $m$ cannot cross leaves of $\sim$). However this is
impossible because then the remap of $V$ (or of $W$) will push $m$ away
from the refixed edge of $V$ (or of $W$) and hence from $G$ resulting in
images of $m$ separated from the $\frac23$-arc created by $m$, a
contradiction with the dynamics of regular critical gaps.

Suppose that $m$ is periodic. Since there are only two cycles of
periodic Fatou gaps of $\sim$ (namely, the orbit of $W$ and the orbit of
$V$), the images of $m$ are contained in $W$ or in $V$. The argument
from above with $m$ instead of $M$ and a suitable choice of $c$
proves that either $m=\ell$ is the refixed edge of $W$,
or $m$ is the refixed edge of $V$ (and hence an edge of $G$).
Clearly, the latter is impossible (the images of $m$ must all be
pairwise disjoint) except when $G=\di$. Thus, if $G\ne \di$, then
$m=\ell$ is the refixed edge of $W$ and $Q=U$. The remaining easy
case $G=\di$ is left to the reader.
\end{proof}

Lemma~\ref{l:coex-quad} gives a preliminary description of
laminations from $\cu^c$. Recall that by Definition~\ref{d:cubioid}
a lamination from $\cu^c$ has at most one rotational set (which then
by necessity is fixed).

\begin{lem}
  \label{l:coex-quad}
  Let $\sim$ be a lamination from $\cu^c$.
  Then $\sim$ coexists with an invariant quadratic gap $U$,
  and if $G$ is the unique rotational set of $\sim$, then $G\subset U$.
  Moreover, either $\sim$ tunes $U$ or no edge of $U$ belongs to $\sim$.
  In the latter case, $U$ can be chosen to be of regular critical type,
  and if $\sim$ is not canonical, then $U$ must be of regular critical type.
  If
  $G$ is of type A or B, then:

  \begin{enumerate}

  \item if $\sim$ is canonical, then $U$ can be chosen to be
  regular critical and weakly tuned by $\sim$;

  \item if $\sim$ is not canonical and has two
  critical sets, a critical gap $V$ attached to $G$ and the second
  critical set $C$ not attached to $G$, then $U$ can be chosen as
  $U(c)$ where $c\subset C$ is \textbf{any} critical chord with non-periodic endpoints.

  \end{enumerate}

\end{lem}

\begin{proof}
By Lemma~\ref{l:pc0}, if $\sim$ has no rotational gap or leaf,
$\sim$ is empty or the canonical lamination of an invariant
quadratic gap as desired. Assume that $\sim$ has a rotational gap or
leaf $G$. By Lemmas ~\ref{l:spec-LvsG1} and ~\ref{l:type1Sie}, and
because $\sim$ comes from $\cu^c$ (and hence has at most one
rotational set), if $G$ is of type D, then it follows that
$\sim=\sim_G$ tunes an invariant quadratic gap $U$, whose major is
one of the two majors of $G$, and $G\subset U$. Hence from now on we
assume that $G$ is of type A, B or C.

First assume that $G$ is finite and has $n$ edges.
By Definition~\ref{d:cubioid}, there exists a cycle $\F$ of Fatou gaps attached to $G$.
Let $G$ be of type B.
Suppose first that $\F$ has two gaps, $V$ and $W$, on that the map $\si$ is two-to-one.
Let $M$ be the major of $G$ that is an edge of $V$, and $\si^k(M)$ be the major of $G$ that is an edge of $W$.
Denote by $N\ne M$ the edge of $V$ such that $\si(N)=\si(M)$, and by $T\ne \si^k(M)$ the edge of $W$ with $\si(T)=\si(\si^k(M))$. We want to find a regular critical
major separating $M$ from $N$. To do so,
consider the model map for $\si^n|_V$ which is $\si_4$ (as always,
the modeling map is the map collapsing all edges of gaps to points).
Clearly, we can find $\si_4$-critical diameters $\ell$, whose orbits are contained in the half-circle bounded by $\ell$ and containing $0$.
By definition, the critical chord $L$ inside $V$ corresponding to $\ell$ is a major of regular critical type of an invariant quadratic gap $U$.

Since the orbits of the endpoints of $\si^k(M)$ and $T$ are
contained in the circle arc of length $\frac23$, whose endpoints are
the endpoints of $L$, these endpoints belong to $U'$. Hence
the edges of $U$ are disjoint from the convex hull $Q$ of $\si^k(M)\cup
T$. This in turn implies that edges of $U$ and leaves of $\sim$ do
not intersect. Indeed, otherwise we can map such intersecting leaves
forward, and their images intersect too (because by the above the
intersecting leaves must be such that their endpoints belong to an
arc of length less than $\frac13$). In the end an edge of $U$ will
map to $L$, a contradiction since we know that $L$ is disjoint from
all leaves of $\sim$. Hence in this case we can always choose $U$ to
be of a regular critical type. Clearly, $G\subset U$.

Now, if $G$ is of type A, then there is only one gap $V$ of $\F$
which does not map forward one-to-one; $V$ is attached to the unique
major edge $M$ of $G$. If $V$ is cubic, then, similar to the above,
we can choose a regular critical chord $L$ inside $V$ so that $L$ is
a major of regular critical type of an invariant quadratic gap $U$.
To show that $\sim$ and $U$ coexist, consider the major $M$ of $G$.
Then there are two edges of $V$ having the same image as $M$. Let
$N$ be one of them chosen so that $L$ does not separate $M$ from $N$.
The existence of $N$ is derived from the following observation:
for a $\si$-critical chord $\ell$ of regular critical type such
that $0\in U'(\ell)$, we have $\frac 13\notin H_{U(\ell)}(\ell)$
or $\frac 23\notin H_{U(\ell)}(\ell)$.
As before, let $Q$ be the convex hull of $M\cup N$. Then
literally repeating the arguments from the previous case, we can
show that $U$ and $\sim$ coexist.

This completes our consideration of the canonical laminations in the
case when $G$ is finite. Thus from now on we may assume that $G$ is
either of type A or of type $B$, and there is a unique Fatou gap $V$
from $\F$ attached to some edge of $G$ such that $\si|_{\bd(V)}$ is
not one-to-one; moreover, in the remaining cases we may assume that
the remap on $V$ is two-to-one. Then clearly there exists a critical
leaf or gap $C$ which is not a gap from $\F$. If $C$ is infinite,
then all the claims follow from Lemma~\ref{l:2fatou}. Hence we may
assume that $C$ is finite; in particular, all vertices of $C$ are
non-periodic.

Choose a critical chord $c$ in $C$. Consider the arc $I$ of
length $\frac23$, one of the two arcs into which $c$ divides
$\uc$. The vertices of $G$ belong to $I$; the bases of $G$ and of
the gaps of $\F$ consist of points of $U'(c)$. We may take $d$ to be
a critical chord of $V$ whose endpoints are not the endpoints of any
leaf of $\sim$ (the basis of $V$ is a Cantor set, so we can choose
$d$ satisfying this property). Clearly $c$ and $d$ satisfy the
conditions of Lemma~\ref{l:lamcoex}, which implies the existence of a
quadratic invariant gap $U$ coexisting with $\sim$ and such
that either $\sim$ tunes $U$, or no edge of $U$ is a leaf of $\sim$.
Let us show that in the latter case $U$ \emph{must} be of regular
critical type.

Indeed, otherwise the major $M$ of $U$ is of periodic type. Since
$M$ is not a leaf of $\sim$, it is contained in a periodic gap
$H$, which is at least a quadrilateral. If $H$ is finite, then, by
\cite{kiwi02}, the remap on $H$ is not the identity map. Thus, $H$ is
the second finite periodic gap of $\sim$ on which the remap is not
the identity, a contradiction with the definition of $\cu^c$. This
implies that $H$ is a periodic Fatou gap. Moreover, by
Lemma~\ref{l:2fatou}, the gap $H$ comes from the orbit of $V$. As in the
proof of Lemma~\ref{l:2fatou}, this yields that $M$ must be an edge
of $G$, a contradiction with $M$ not being a leaf of $\sim$.

A similar argument holds in the case, where $G$ is a Siegel gap.
Let $d$ be a critical edge of $G$. There is some other
critical leaf or gap $C$. Let $c$ be a critical chord in $C$.  As
before, $c$ may be chosen to have non-periodic endpoints and
$G\subset U(c)$. Since no leaves of $\sim$ other than $d$ intersect
$d$, then $c$ and $d$ satisfy the conditions of
Lemma~\ref{l:lamcoex}.
\end{proof}

We are ready to prove the Main Theorem stated in the Introduction.
The following statement makes it more precise (we refer to the
notation introduced in the Main Theorem). Let $\sim$ be a lamination
from $\cu^c$. Then by Lemma~\ref{l:coex-quad} $\sim$ coexists with a
quadratic invariant gap $U$. The gap $U$ in the Main Theorem can be
chosen as in Lemma~\ref{l:coex-quad}.

\begin{thm}
\label{t:cormin-spec} Assume the conditions of
Lemma~\ref{l:coex-quad} and adopt the notation from its conclusion.
If the $\psi_U$-image of the major $M$ of $U$ does not eventually
map $($by $\si_2)$ to a periodic Fatou gap of $\asymp=\psi_U(\sim)$,
then case $(2)$ of
the Main Theorem holds.
This is also the case, when $U$ is of periodic type and the lamination $\sim$ is not the canonical lamination of a finite invariant rotational gap or leaf of type A or B.
\end{thm}

\begin{proof}[Proof of the Main Theorem and of Theorem \ref{t:cormin-spec}]
Suppose that $\sim$ is not a canonical lamination of an invariant
quadratic gap. Then, by Theorem~\ref{t:fxpt}, there exists an
invariant rotational gap or leaf $G$. Suppose that $G$ is of type D
(finite or Siegel). Then, by Lemma~\ref{l:spec-LvsG1} and
Lemma~\ref{l:type1Sie}, the lamination $\sim$ must be the canonical
lamination of $G$. Choose a major $M$ of $G$. It follows that the
same major $M$ defines also an invariant quadratic gap $U$. It is
easy to see that then $\sim$ tunes the canonical lamination $\sim_U$
according to an appropriate quadratic invariant lamination from
$\car^c$, which corresponds to case (2). From now on, we may assume
that $G$ is not of type D.

By Lemma~\ref{l:coex-quad}, the lamination $\sim$ coexists with a
quadratic invariant stand-alone gap $U$, either $\sim$ tunes $U$ or
no edge of $U$ belongs to $\sim$ and $U$ can be chosen to be of
regular critical type (moreover, if $\sim$ is not canonical, then $U$ must
be of regular critical type). Furthermore, $G\subset U$. Suppose
that the map $\psi_U$ projects the restriction of $\sim$ onto $U$ to
a quadratic invariant lamination $\asymp=\psi_U(\sim)$. By
Definition~\ref{d:wtunegap-i}, the lamination $\sim$ weakly tunes
$U$ according to the lamination $\asymp$.
Since $\sim$ has a unique rotational set, so does $\asymp$.
By Proposition~\ref{p:coremin-qua}, the lamination $\asymp$ comes from
the Combinatorial Main Cardioid $\car^c$.

If $U$ is of periodic type and its canonical lamination is tuned by
$\sim$ then, if the gap $V(U)$ is of period greater than 1, it
cannot be tuned by $\sim$ as this would create a rotational set of
period greater than 1. Now, the only cases when $V(U)$ is of period
1 are when $U=\fg_a, V(U)=\fg_b$, or $U=\fg_b, V(U)=\fg_a$. In the
former case, if $\sim$ is not empty inside $U$, then again $V(U)$
cannot be non-trivially tuned by $\sim$ as this would create two
rotational sets of $\sim$. Suppose now that $\sim$ is empty inside
$U$. Then it may happen that $V(U)$ is non-trivially tuned by
$\sim$. Similarly to the above, this tuning must be according to
some quadratic lamination $\asymp$ from $\car^c$. In that case we
simply declare that $\widehat U=V(U)=\fg_b$ and $V(\widehat U)=
\fg_a$. Clearly, this is possible, and with this choice of the gap
tuned by $\sim$, the lemma holds. The case $U=\fg_b, V(U)=\fg_a$ is
similar.

Now we need to prove the remaining claims of the theorem. By
Lemma~\ref{l:coex-quad}, to see whether $\sim$ tunes $U$, we need to see
whether the major $M$ of $U$ belongs to $\sim$. Suppose that the point
$\psi_U(M)$ does not eventually map (by $\si_2$) to a periodic
infinite gap of $\asymp$. Then well-known properties of quadratic
laminations from the Combinatorial Main Cardioid $\car^c$ imply that
$\psi_U(M)$ is separated from the rest of the circle by a sequence
of leaves of $\asymp$. Hence $M$ is the limit of appropriately
chosen leaves of $\sim$, and so $M$ itself is a leaf of $\sim$.
Thus, if $\psi_U(M)$ does not eventually map (by $\si_2$) to a
periodic infinite gap of $\asymp$, then $\sim$ tunes $U$.

Finally, let us prove the last claim of the theorem. We need to
prove that if $U$ is of periodic type and $\sim$ is not canonical,
then case (2) must hold. However, this immediately follows from
Lemma~\ref{l:coex-quad} as in the case of weak tuning and
non-canonical lamination this lemma states that $U$ must be of
regular critical type. This completes the proof.
\end{proof}

The statement of the Main Theorem is somewhat involved.
However it leads to a more explicit description if one thinks of constructing
a non-empty lamination $\sim$ from $\cu^c$. Indeed, for definiteness
assume that $\sim$ has a finite rotational gap $G$. Observe that for
canonical laminations of type D the explanation as how $\sim$ fits
into the description from the Main Theorem is given in
the proof. Otherwise, just like in the arguments of some of our
theorems, consider both critical sets of $\sim$. One of them is
attached to $G$. The other one can be either (a) a vassal gap of
some invariant quadratic gap $U$ of periodic type, or (b) a critical
leaf which is a major of regular critical type of some quadratic invariant
gap $U$, or (c) the same as the first one (canonical lamination of
type A), or (d) an infinite gap-preimage of the first one (canonical
laminations of type B or C), or (e) a finite gap-preimage of $G$.

In cases (a) or (b) the lamination $\sim$ tunes the canonical
lamination of $U$ according to a lamination from the Main Cardioid.
In the other cases the basis of the second critical set contains
endpoints of a critical leaf which is itself a major of regular
critical type of some quadratic invariant gap $U$; moreover, all
other edges of $U$ are also present as diagonals (but not as edges)
of other gaps of $\sim$. The construction of such $\sim$ can be
viewed as a three step process: first, we take the canonical
lamination of an invariant quadratic gap $U$, then $U$ is tuned
according to a quadratic lamination from the Main Cardioid, and then
finally edges of $U$ and their preimages are erased giving rise to
$\sim$ (whether we get a lamination described in (c), (d) or (e)
above depends on the relation between the major of $U$, the gap $G$
and gaps of the canonical lamination of $G$ attached to $G$).

In conclusion, we prove Corollary~\ref{c:othercu-i} from the
Introduction, which allows for a shorter definition of laminations
from the Combinatorial Main Cubioid $\cu^c$.

\begin{proof}[Proof of Corollary \ref{c:othercu-i}]
The ``if" part of the claim follows immediately from definitions. To
prove the ``only if" part of this corollary one simply has to go over
different types of laminations listed in
Theorem~\ref{t:cormin-spec} (or in our explanation right after this
theorem). Indeed, first we observe that by definition if $\sim$
belongs to $\cu^c$ then it has at most one rotational periodic
(hence fixed) set. Now, consider the second part of the claim. It is
obvious for canonical laminations of quadratic invariant gaps or for
canonical laminations of finite gaps of type D. In the case when the
lamination $\sim$ is obtained as described in Case (1) of
Theorem~\ref{t:cormin-spec} --- or, equivalently, in cases (c), (d)
or (e) above --- the only periodic leaves of $\sim$ are edges of the
periodic rotational gap $G$ (and only in the case when $G$ is
finite), so the claim follows.
Finally, if Case (2) of Theorem~\ref{t:cormin-spec} applies, then, in addition to the edges of $G$, the lamination $\sim$ may also have periodic leaves $\ell$ that form the orbit of a major $M$ of periodic type generating an invariant quadratic gap $U$ from Theorem~\ref{t:cormin-spec}.
However, in that case, there must exist an infinite gap attached to each such leaf $\ell$ that itself belongs to the orbit of the vassal gap attached to $M$.
Thus the claim holds in this case too.
\end{proof}

\bibliographystyle{amsalpha}

\begin{thebibliography}{1}



\bibitem[ALM00]{alm00} (MR1807264) \newblock L. Alseda, J. Llibre, M.
    Misiurewicz,
    \newblock \emph{Combinatorial Dynamics and Entropy in Dimension One},
    \newblock World
    Scientific (Advanced Series in Nonlinear Dynamics, vol. 5), Second
    Edition (2000)


\bibitem[BCO08]{bco08} (MR2737795) \newblock A. Blokh, C. Curry, L. Oversteegen,
    \newblock \emph{Locally connected models for Julia sets},
    \newblock Advances in
    Mathematics, \textbf{226} (2011), 1621--1661.

\bibitem[BFMOT10]{bfmot10} (MR3087640) \newblock A. Blokh, R. Fokkink, J. Mayer,
    L.
    Oversteegen, E. Tymchatyn,
    \newblock \emph{Fixed point theorems in plane
    continua with applications},
    \newblock Memoirs of the American Mathematical
    Society, \textbf{224} (2013), no. 1053

\bibitem[BL02]{blolev02a} (MR1889565) \newblock A.~Blokh, G.~Levin,
\newblock {\em Growing
    trees, laminations and the dynamics on the Julia set},
    \newblock Ergod. Th.
    and Dynam. Sys., \textbf{22} (2002), 63--97.

\bibitem[BMMOP06]{bmmop06} (MR2216120) \newblock A. Blokh, J. Malaugh, J. Mayer,
    L.
    Oversteegen, D. Parris,
    \newblock {\em Rotational subsets of the circle under
    $z^n$},
    \newblock Topology and its Appl., \textbf{153} (2006), 1540--1570.

\bibitem[BMOV11]{bmov11}  (MR3074377)
\newblock A. Blokh, D. Mimbs, L. Oversteegen, K. Valkenburg,
\newblock \emph{Laminations in the language of leaves},
\newblock Trans. of the Amer.
Math. Soc. \textbf{365} (2013), 5367--5391.





\bibitem[BO10]{bo10}  (MR2661256) \newblock A. Blokh, L. Oversteegen,
\newblock \emph{Monotone images of Cremer Julia sets},
\newblock Houston Journal of Mathematics, \textbf{36} (2010), 469--476.

\bibitem[BOPT11]{bopt11} (MR3289905) \newblock A. Blokh, L. Oversteegen, R. Ptacek, V.
Timorin,
\newblock \emph{Dynamical cores of topological polynomials},
in: Proceedings for the Conference
``Frontiers in Complex Dynamics (Celebrating John Milnor's 80th birthday)''

\bibitem[BOPT13]{bopt13} (MR3246159) \newblock A. Blokh, L. Oversteegen, R. Ptacek, V.
Timorin,
\newblock \emph{The Main Cubioid},
\newblock Nonlinearity \textbf{27} (2014),
1879--1897.


\bibitem[BH01]{BH01} (MR1844629)
\newblock X. Buff, C. Henriksen,
\newblock \emph{Julia Sets in Parameter Spaces},
\newblock Commun. Math. Phys. \textbf{220} (2001), 333--375.

\bibitem[Car913]{car913} (MR1511737) \newblock C. Carath\'eodory,
\newblock \emph{\"Uber die Begrenzung einfach zusammenh\"angender Gebiete
(German)},
\newblock Math. Ann. \textbf{73} (1913), 323-–370.

\bibitem[CG02]{cg} (MR1230383) \newblock L. Carleson, T.W. Gamelin,
\newblock \emph{Complex dynamics},
\newblock Springer (1992).


\bibitem[DH84]{hubbdoua85a} (MR0762431) \newblock A. Douady, J. H. Hubbard,
\newblock \emph{\'Etude dynamique
des polyn\^omes complexes} I,
\newblock Publications Math\'ematiques
d'Orsay \textbf{84-02} (1984).

\bibitem[DH85]{hubbdoua85b}  (MR0812271) \newblock A. Douady, J. H. Hubbard,
\newblock \emph{\'Etude dynamique
des polyn\^omes complexes} II,
\newblock Publications Math\'ematiques
d'Orsay \textbf{85-04} (1985).

\bibitem[EY99]{EY99} (MR1681808) \newblock A. Epstein, M. Yampolsky,
\newblock \emph{Geography of the Cubic Connectedness Locus: Intertwining Surgery},
\newblock Ann. Sci. \'Ec. Norm. Sup., \textbf{32} (1999), 151--185.


\bibitem[Gau14]{gau14} (MR3263916) \newblock T. Gauthier,
\newblock \emph{Higher bifurcation currents, neutral
cycles, and the Mandelbrot set},
\newblock Indiana Univ. Math. J. \textbf{63}
(2014), 917-–937.

\bibitem[GM93]{gm93} (MR1209913) \newblock L. Goldberg and J. Milnor,
\newblock \emph{Fixed points of polynomial maps. II. Fixed point portraits},
\newblock Ann. Sci. \'Ecole Norm. Sup. (4), \textbf{26} (1993), 51--98.


\bibitem[Kiw02]{kiwi02} (MR1873015) \newblock J. Kiwi,
\newblock \emph{Wandering orbit portraits},
\newblock Trans. of the
Amer. Math. Soc., \textbf{254} (2002), 1473--1485.

\bibitem[Kiw04]{kiwi97} (MR2054016)
\newblock J. Kiwi,
\newblock \emph{$\mathbb R$eal laminations and the
topological dynamics of complex polynomials},
\newblock Advances in Mathematics,
\textbf{184} (2004), 207--267.



\bibitem[McM07]{mcm07} (MR1765082) \newblock C. McMullen,
\newblock \emph{The Mandelbrot set is universal},
\newblock in: The Mandelbrot Set, Theme and Variations, ed. T. Lei, Cambridge
U.K. Cambridge Univ. Press. Revised (2007), 1-–18.

\bibitem[Mil93]{miln93} (MR1246482)
\newblock J. Milnor,
\newblock \emph{Geometry and Dynamics of Quadratic Rational Maps},
\newblock Experimental Math., \textbf{2} (1993), 37--83.

\bibitem[Mil00]{miln00} (MR2193309) \newblock J. Milnor,
\newblock \emph{Dynamics in One Complex Variable},
\newblock Annals of Mathematical Studies \textbf{160}, Princeton (2006).

\bibitem[Mil09]{M} (MR2508263) \newblock J. Milnor,
\newblock \emph{Cubic polynomial maps with periodic critical orbit I},
\newblock in:
``Complex Dynamics, Families and Friends'', ed. D. Schleicher, A.K.
Peters (2009), 333--411.

\bibitem[MP92]{MP92}
\newblock J. Milnor, A. Poirier,
\newblock \emph{Hyperbolic Components in Spaces of Polynomial Maps},
\newblock preprint, arXiv:math/9202210

\bibitem[MT88]{mt88} (MR0970571) \newblock J. Milnor, W. Thurston,
\newblock \emph {On iterated maps of the interval},
\newblock in
Dynamical systems, Lecture Notes in Math. \textbf{1342}, Springer,
Berlin  (1988), 465--563

\bibitem[Mis79]{mis79} (MR0542778) \newblock M. Misiurewicz,
\newblock \emph{Horseshoes for mappings of the interval},
\newblock Bull. Acad. Pol. Sci., Ser. sci. math., astr. et phys. \textbf{27} (1979), 167--169.

\bibitem[MS80]{ms80} (MR0579440) \newblock M. Misiurewicz, W. Szlenk,
\newblock \emph{Entropy of piecewise monotone
mappings},
\newblock Studia Math. \textbf{67} (1980), 45--63.

\bibitem[PT09]{PT09} (MR2508264) \newblock C. L. Petersen, Tan Lei,
\newblock \emph{Analytic Coordinates Recording Cubic Dynamics}.
\newblock In: ''Complex
Dynamics: Families and Friends'', ed. Dierk Schleicher, Wellesley
Massachusetts: A. K. Peters, Limited (2009), 413--449.

\bibitem[PRT]{PRT}
\newblock C. L. Petersen, P. Roesch and Tan Lei,
\newblock \emph{Parabolic slices on the
boundary of $\mathcal H$},
\newblock work in progress.

\bibitem[Roe06]{R06}
\newblock P. Roesch,
\newblock \emph{Hyperbolic components of polynomials with a fixed critical point of
maximal order},
\newblock Pr\'epublications du laboratoire Emile Picard, Univ. Paul Sabatier,
U.F.R. M.I.G., \textbf{306} (2006)

\bibitem[Thu85]{thu85} (MR2508255) \newblock W.~Thurston,
\newblock {\em On the geometry and dynamics of
iterated rational maps} (1985),
\newblock in: ``Complex dynamics: Families
and Friends'', ed. by D. Schleicher, A K Peters (2008), 1--108.

\bibitem[You81]{you81} (MR0590412)
\newblock L.-S. Young,
\newblock \emph{On the prevalence of horseshoes},
\newblock Trans. Amer.
Math. Soc. \textbf{263} (1981), 75--88.

\bibitem[Zak99]{Z99} (MR1736986)
\newblock S. Zakeri,
\newblock \emph{Dynamics of cubic Siegel polynomials},
\newblock Comm. Math.
Phys. \textbf{206} (1999), 185--233.

\end{thebibliography}

\medskip
Received May 24, 2013; revised July 16, 2014; second revision October 20, 2014; third revision January 14, 2016.
\medskip

\end{document}